\documentclass[11pt,a4paper,leqno]{amsart}

\usepackage[latin1]{inputenc}
\usepackage[T1]{fontenc}
\usepackage{amsfonts}
\usepackage{amsmath}
\usepackage{amssymb}
\usepackage{eurosym}
\usepackage{mathrsfs}
\usepackage{palatino}
\usepackage{color}
\usepackage{esint}
\usepackage{url}
\usepackage{verbatim}
\usepackage{enumerate}
\usepackage[dvipsnames]{xcolor}

\usepackage{caption,float}
\usepackage{placeins}
\usepackage{afterpage}

\usepackage{subfig}

\usepackage[pagebackref,hypertexnames=false, colorlinks, citecolor=blue, linkcolor=blue, urlcolor=blue]{hyperref}
\usepackage{amsrefs}

\newfloat{fig}{tb}{lop}[section]
\DeclareCaptionLabelFormat{Ucase}{Figure~#2}
\usepackage{graphicx}
\graphicspath{ {./pics_blc/} }

\newcommand{\R}{\mathbb{R}}
\newcommand{\C}{\mathbb{C}}

\newcommand{\Z}{\mathbb{Z}}
\newcommand{\E}{\mathbb{E}}

\newcommand{\calF}{\mathcal{F}}

\newcommand{\calR}{\mathcal{R}}

\newcommand{\diam}{\operatorname{diam}}

\newcommand{\vare}{\varepsilon}

\newcommand{\CZO}{\textup{CZO}}

\newcommand{\CZ}{\textup{CZ}}

\numberwithin{equation}{section}

\newcommand{\ud}[0]{\,\mathrm{d}}

\newcommand{\dist}[0]{\operatorname{dist}}

\newcommand{\abs}[1]{|#1|}
\newcommand{\babs}[1]{\big|#1\big|}
\newcommand{\Babs}[1]{\Big|#1\Big|}

\newcommand{\Norm}[2]{\|#1\|_{#2}}
\newcommand{\bNorm}[2]{\big\|#1\big\|_{#2}}

\newcommand{\ave}[1]{\langle #1\rangle}
\newcommand{\bave}[1]{\big\langle #1\big\rangle}

\newcommand{\osc}[0]{\operatorname{osc}}

\newcommand{\BMO}[0]{\operatorname{BMO}}

\newcommand{\supp}[0]{\operatorname{spt}}
\newcommand{\loc}[0]{\operatorname{loc}}

\newcommand{\eps}[0]{\varepsilon}

\newcommand{\ch}[0]{\operatorname{ch}}

\newcommand{\calD}[0]{\mathcal{D}}

\newcommand{\wt}[1]{{\widetilde{#1}}}

\swapnumbers
\theoremstyle{plain}
\newtheorem{thm}[equation]{Theorem}
\newtheorem{lem}[equation]{Lemma}
\newtheorem{prop}[equation]{Proposition}

\theoremstyle{definition}
\newtheorem{defn}[equation]{Definition}

\theoremstyle{remark}

\newtheorem{rem}[equation]{Remark}

\title{Off-diagonal estimates for bilinear commutators}

\author{Tuomas Oikari}
\address[T.O.]{Department of Mathematics and Statistics, University of Helsinki, P.O.B. 68, FI-00014 University of Helsinki, Finland}
\email{tuomas.v.oikari@helsinki.fi}

\makeatletter
\@namedef{subjclassname@2010}{%
	\textup{2010} Mathematics Subject Classification}
\makeatother

\thanks{T. Oikari was supported by the Academy of Finland project No. 306901, by the Finnish Centre of Excellence in Analysis and Dynamics Research project No. 307333, and by the three-year research grant of the University of Helsinki No. 75160010.
}

\begin{document}	
\begin{abstract} We find a minimal notion of non-degeneracy for bilinear singular integral operators $T$ and identify testing conditions on the multiplying function $b$ that characterize the $L^p\times L^q\to L^r,$ $1<p,q<\infty$ and $r>\frac{1}{2},$  boundedness of the bilinear commutator $[b,T]_1(f,g) = bT(f,g) - T(bf,g).$ Our arguments cover almost all arrangements of the integrability exponents $p,q,r,$ with a single open problem presented in the end.
Additionally, the arguments extend to the multilinear setting.
\end{abstract}
\maketitle
\section{Introduction}
The study of commutator estimates have their roots in the work of Nehari \cite{Nehari1957} where the boundedness of the commutator
of the Hilbert transform and a multiplying function $b,$ 
\[
[b,H]f(x) = b(x)Hf(x)-H(bf)(x),\qquad Hf(x) = p.v.\int_{\R}f(x-y)\frac{\ud y}{y},
\]  was characterized by Hankel operators. Later, Coifman and Rochberg and Weiss extended Nehari's result by real analytic methods and showed that
\begin{equation}\label{eq:commutatorbmo}
\|b\|_{\BMO} \lesssim \sum_{j=1}^d\|[b,\calR_j]\|_{L^p(\R^d) \to L^p(\R^d)} \lesssim \|b\|_{\BMO} :=\sup_I \fint_I  \abs{b-\ave{b}_I}, \qquad p \in (1,\infty),
\end{equation}
where the supremum is taken over all cubes $I \subset \R^d$ and $\ave{b}_I = \frac{1}{|I|} \int_I b$. The upper bound in \eqref{eq:commutatorbmo} was proved in \cite{CRW} for general Calderón-Zygmund operators $T$, while with the lower bounds they worked with the Riesz transforms $\calR_j;$ commutator upper bounds are usually valid for all Calderón-Zygmund operators (CZO), while the lower bounds require some non-degeneracy. The lower bound in \eqref{eq:commutatorbmo} was improved separately by both Janson \cite{Janson1978} and Uchiyama \cite{Uchiyama1978} to $\|b\|_{\BMO} \lesssim \Norm{[b,T]}{{L^p(\R^d) \to L^p(\R^d)}}$ under certain non-degeneracy assumptions on the kernel of $T$ which encompass any single Riesz transform (in contrast with \eqref{eq:commutatorbmo} involving all the $d$ Riesz transforms). Janson's proof also gives the following off-diagonal characterization of the boundedness of the commutator
\begin{align}\label{eq:commutatoralpha}
	\Norm{[b,T]}{L^p\to L^q}\sim \Norm{b}{\dot C^{0,\alpha}} \sim \sup_Q\ell(Q)^{-\alpha}\fint_Q\abs{b-\ave{b}_Q}, \qquad \alpha :=d\Big(\frac{1}{p}-\frac{1}{q}\Big),
\end{align}
when $1<p<q<\infty.$
The off-diagonal characterizations in the case $1<q<p<\infty$ turned out to be harder and was only recently solved by the  approximate weak factorization (awf) argument in  Hyt\"onen  \cite{HyCom},
\begin{align}\label{eq:commutatorS}
		\Norm{[b,T]}{L^p\to L^q}\sim \Norm{b}{\dot L^s} := \inf_{c\in\C}\Norm{b-c}{L^s},\quad \frac{1}{q} = \frac{1}{s} + \frac{1}{p},\quad 1<q<p<\infty.
\end{align}

\vspace{0.2cm}
Commutator estimates imply factorization results for Hardy spaces, see \cite{CRW}, they have applications in partial differential equations by compensated compactness, div-curl lemmas, see  \cite{CLMS1993}, and they have been crucial in the recent investigations to the Jacobian problem, see Lindberg \cite{Lindberg2017} and \cite{HyCom}.
\vspace{0.2cm}

The awf argument is strong in that it gives a unified approach to all of the three cases, those on the lines \eqref{eq:commutatorbmo}, \eqref{eq:commutatoralpha}, \eqref{eq:commutatorS}, in that it works for many singular integrals with kernels satisfying only minimum non-degeneracy assumptions, and in that it is flexible enough to grant e.g. multi-parameter and multilinear extensions. 
For the multi-parameter variants of the awf argument see Airta, Hyt\"{o}nen, Li, Martikainen, Oikari \cite{AirHytLiMartOik2020offdiag} and Oikari \cite{Oik2020(1)},
where, respectfully, the commutators
\begin{align}\label{line2}
\big[T_2,[T_1,b]\big],\big[b,T\big]:L^{p_1}(\R^{d_1};L^{p_2}(\R^{d_2}))\to L^{q_1}(\R^{d_1};L^{q_2}(\R^{d_2}))
\end{align}
were treated. On the line \eqref{line2} $1<p_1,p_2,q_1,q_2<\infty,$ $T_i$ is a one-parameter CZO on $\R^{d_i},$ for $i=1,2,$ and $T$ is a bi-parameter CZO on $\R^{d_1+d_2}.$ 
The adaptability of the awf argument to the bi-parameter settings was not effortless and for both commutators on the line \eqref{line2} the characterization of some cases is still open.

In this article we extend the awf argument to the bilinear setting and study the two commutators
$$
[b,T]_1(f,g) = bT(f,g)-T(bf,g),\qquad [b,T]_2(f,g) = bT(f,g)-T(f,bg)
$$  as mappings $L^p\times L^q\to L^r$ for $r>\frac{1}{2}$ and $1<p,q<\infty.$ Our cases separate accordingly to the following three conditions
\begin{align}\label{defn:classes}
\underset{\emph{sub-diagonal}}{\frac{1}{r} < \frac{1}{p}+\frac{1}{q}},\qquad \underset{\emph{diagonal}}{	\frac{1}{r} = \frac{1}{p}+\frac{1}{q}},\qquad  \underset{\emph{super-diagonal}}{\frac{1}{r} > \frac{1}{p}+\frac{1}{q}}.
\end{align}
If $r>1$ then we are in the Banach range of exponents and if $r\leq1$ then we are in the quasi Banach range of exponents.
In Chaffee \cite{ChaffeeLucas2016Cobm} the necessity of $b\in\BMO$ on the diagonal in the Banach range of exponents was shown with kernels expandable locally as a Fourier series. A unified approach to the diagonal and sub-diagonal cases was given in Guo, Lian and Wu \cite{guo2017unified}, which covers the diagonal in the whole quasi Banach range, however on the sub-diagonal they only treat the linear case.

 In addition to involving new cases, the super-diagonal case is new in the bilinear setting, our results generalize previous work: the definition of non-degeneracy is weaker than those supposed in \cite{LiWick2017}, \cite{kuffner2018weak}, \cite{ChaffeeLucas2016Cobm}, \cite{guo2017unified} and \cite{Li2020multilinear}; the awf argument allows us to consider complex valued functions $b,$ whereas \cite{Li2020multilinear} was limited to the real valued case; the full quasi-Banach range is reached in the diagonal and sub-diagonal cases, whereas \cite{ChaffeeLucas2016Cobm} is limited to the Banach range; and in that the awf argument encompasses bilinear CZOs with both  Dini and rough kernels. Lastly, due to us studying the quasi Banach range, the arguments involve additional twists absent from previous research articles. Our full results are recorded as theorems \ref{thm:diag/subdiag/lb},  \ref{thm:diag/ub}, \ref{thm:subdiag/ub} and \ref{thm:superdiag}, the following being a condensed version.
\begin{thm}\label{thm:main} Let $b\in L^1_{\loc}(\R^d;\C),$ let $T$ be a non-degenerate bilinear Calderón-Zygmund operator, let $\frac{1}{2}<r<\infty$ and $1<p,q<\infty$.  Then, for $i=1,2,$ there holds that 
	\begin{align*}
	\Norm{[b,T]_i}{L^p\times L^q\to L^r} \sim	\begin{cases}
		\Norm{b}{\BMO}, & \mbox{if}\quad \frac{1}{r} = \frac{1}{p} + \frac{1}{q} \\
		\Norm{b}{\dot C^{\alpha,0}},\quad \alpha = d\big( (\frac{1}{p}+\frac{1}{q})-\frac{1}{r}\big), &\mbox{if}\quad \frac{1}{r} < \frac{1}{p} + \frac{1}{q} \\
		\Norm{b}{\dot{L}^s},\quad \frac{1}{r} = \frac{1}{s}+\frac{1}{p}+\frac{1}{q},  & \mbox{if}\quad \frac{1}{r} > \frac{1}{p} + \frac{1}{q},\quad r\geq 1.
		\end{cases}
	\end{align*}
\end{thm}

\subsubsection*{Acknowledgements}
I thank Henri Martikainen for helpful discussions and comments.

\section{Definitions and preliminaries}

\subsection{Basic notation}
We let $\Sigma = \Sigma(\R^d)$ denote the linear span of indicator functions of cubes on $\R^d.$ Similarly we denote $L^1_{\loc}(\R^d) = L^1_{\loc},$ $\int_{\R^d} = \int,$ and so on, mostly leaving out the ambient space if this information is obvious.
We denote averages with
$
\langle f \rangle_A = \fint_A f:= \frac{1}{|A|} \int_A f,
$
where $|A|$ denotes the Lebesgue measure of the set $A$. The indicator function of a set $A$ is denoted by $1_A$.

In this paper we study the  $L^p\times L^q\to L^r$ boundedness and hence it is useful to denote 
$
\sigma(p,q)^{-1} = p^{-1}+q^{-1};
$
then,  H{\"o}lder's inequality writes as $
\Norm{fg}{L^{\sigma(p,q)}} \leq \Norm{f}{L^{p}} \Norm{g}{L^{q}}.$

Lastly, we denote $A \lesssim B$, if $A \leq C B$ for some constant $C>0$ depending only on the dimension of the underlying space, on integration exponents and on other absolute constants appearing in the assumptions that we do not care about.
Then  $A \sim B$, if $A \lesssim B$ and $B \lesssim  A.$ Subscripts on constants ($C_{a,b,c,...}$) and quantifiers ($\lesssim_{a,b,c,...}$) signify their dependence on those subscripts.

\subsection{Bilinear singular integrals}
We denote the diagonal with
\begin{align*}
	\Delta =  \{(x,y,z)\in ( \R^d)^3:  x=y=z\}
\end{align*}
and say that a mapping $K:(\R^d)^3\setminus \Delta \to \C$ is a bilinear Calderón-Zygmund kernel if it satisfies
the size estimate:
	\begin{align}\label{kernel:size}
	\abs{K(x,y,z)} \leq C_K \big(\abs{x-y}+\abs{x-z}\big)^{-2d},
	\end{align}
and the regularity estimate:
	\begin{equation}\label{kernel:regularity}
	\abs{G(x,y,z) - G(x',y,z)} \leq \omega\big(\frac{\abs{x-x'}}{\abs{x-y}+\abs{x-z}}\big) (\abs{x-y}+\abs{x-z})^{-2d}
	\end{equation}
	for $G\in \{K,K^{1*},K^{2*}\},$ whenever $\abs{x-x'} \leq \frac{1}{2}\max(\abs{x-y},\abs{x-z}).$ 	
Here the function $\omega$ is increasing, subadditive, and such that $\omega(0) = 0$ and $\Norm{\omega}{\textup{Dini}} =\int_0^{1}\omega(t)\frac{\ud t}{t} < \infty.$ 
We also assume that the appearing constants $C_K,\Norm{\omega}{\textup{Dini}}$ are the best possible. We denote the collection of all such kernels with $\CZ(2,d,\omega)$ and associated to this class is the norm $\Norm{K}{\CZ(2,d,\omega)} =  C_K + \Norm{\omega}{\textup{Dini}}.$

\begin{defn}\label{defn:SIO:variable} A bilinear operator $T:\Sigma^2\to L^1_{\loc}$ is said to be a (variable kernel) bilinear SIO, if there exists a bilinear kernel $K\in\CZ(2,d,\omega)$ so that
	\[
	\langle T(f_1,f_2),g\rangle = \int\int\int K(x,y,z)  f_1(y)f_2(z) g(x)\ud y\ud z\ud x
	\]
	for all triples $f_1,f_2,g\in\Sigma$ satisfying $\cap_{i=1}^2\supp(f_i)\cap\supp(g) = \emptyset.$
\end{defn}

\begin{defn}\label{defn:SIO:rough} A bilinear operator $T_{\Omega}:\Sigma^2\to L^1_{\loc}$ is said to be a rough bilinear SIO whenever it is well-defined as
	\begin{align*}
	T_{\Omega}(f_1,f_2)(x) = \lim_{\varepsilon\to 0}\int\int_{\max(\abs{x-y},\abs{x-z})>\varepsilon} K_{\Omega}(x,y,z)f_1(y)f_2(z)\ud y\ud z,
	\end{align*}
where 
$$
K_{\Omega}(x,y,z) = \frac{\Omega((x-y,x-z)')}{\abs{(x-y,x-z)}^{2d}},\qquad\Omega(h') = \Omega(h/\abs{h}),\qquad\Omega\in L^1(\mathbb{S}^{2d-1}).
$$
\end{defn}

\subsection{Truncations of bilinear SIOs}\label{sec:trunc}
 We let $K\in \CZ(2,d,\omega)$ and define the truncated operator $T_{\varepsilon}$  as 
 \begin{align}\label{trunc:0}
 	T_{\varepsilon}(f,g)(x) = \iint_{\substack{\max(\abs{x-y},\abs{x-z})>\varepsilon}} K(x,y,z)f(y)g(z)\ud y\ud z,\qquad f,g\in\Sigma.
 \end{align}
 
A particular case of Cotlar's inequality in the bilinear setting states that 
\begin{align}\label{cotlar}
 T^*(f,g) =	\sup_{\varepsilon>0}\abs{T_{\varepsilon}(f,g)} \lesssim \abs{T(f,g)} + MfMg,\qquad K\in \CZ(2,d,\abs{\cdot}^{\delta}),\quad \delta\in(0,1),
\end{align}
where $M$ is the Hardy-Littlewood maximal operator. For \eqref{cotlar}, see e.g. Grafakos, Torres \cite{GT}.
Since $T,M$ are bounded, it follows that $\sup_{\varepsilon>0}\Norm{T_\varepsilon}{L^p\times L^q\to L^{\sigma(p,q)}}< \infty$ for $1<p,q<\infty.$
For rough kernels Cotlar's inequality was not found. However to achieve a uniform bound on the truncations we need less. It was very recently shown in Theorem 1.1. of \cite{grafakos2020l2timescdotstimes} that under the assumptions $\Omega \in L^q(\mathbb{S}^{2d-1}),$ for some $q>\frac{4}{3},$ and $\int_{\mathbb{S}^{2d-1}}\Omega = 0,$ there holds that 
$
\Norm{T^*_{\Omega}}{L^2\times L^2 \to L^1} \lesssim \Norm{\Omega}{L^{q}(\mathbb{S}^{2d-1})}
$ 
and this implies a uniform bound of the desired type.

\subsection{Boundedness assumptions on $T$}
The majority of this paper is devoted to proving commutator lower bounds and there we do not need any boundedness assumptions on the operator $T$ --  only some non-degeneracy assumptions on the kernel $K$ of $T$ and  some very weak regularity conditions, see Section \ref{sec:nondeg} below. For the upper bounds we impose some boundedness on $T$ and this will vary depending whether we are on the sub-diagonal, diagonal or the super-diagonal case.

\begin{defn}\label{defn:CZO:multilinear} A bilinear Calderón-Zygmund operator refers to a bilinear singular integral $T$ that enjoys some boundedness properties, and in this paper this will be one of the following,
	\begin{enumerate}
		\item $\Norm{T}{L^p\times L^q\to L^{\sigma(p,q)}}<\infty$ for a single tuple of exponents $p,q\in(1,\infty),$
		\item $\Norm{T}{L^p\times L^q\to L^{\sigma(p,q)}}<\infty$ for all tuples of exponents $p,q\in(1,\infty),$
		\item $\sup_{\vare>0}\Norm{T_{\varepsilon}}{L^p\times L^q\to L^{\sigma(p,q)}}<\infty$ for a single tuple of exponents $p,q\in(1,\infty).$
	\end{enumerate}
\end{defn}
We will assume $(1)$ for the diagonal upper bound, $(2)$ for the super-diagonal upper bound and $(3)$ for the sub-diagonal upper bound.
\begin{rem}	
	The bilinear Riesz transforms, one of which is
	\begin{align}\label{eq:riesz}
	\calR_i(f,g)(x) = p.v.\int\int \frac{x_i-y_i}{(\abs{x-y}+\abs{x-z})^{2d+1}}f(y)g(z)\ud y\ud z,
	\end{align} satisfy each of the above boundedness properties in Definition \ref{defn:CZO:multilinear}, and consequently, Theorem \ref{thm:main} is valid as stated with $T=\calR_i.$
\end{rem}
\subsection{Bilinear non-degeneracy}\label{sec:nondeg} 
We first recall the definition of non-degeneracy for linear kernels.
\begin{defn} A kernel $K:\R^d\times\R^d\setminus\Delta\to \C$ is said to be non-denegerate, if given $y\in\R^d$ and $r>0$ there exists a point $x$ so that
	\[
	\abs{x-y}>r,\qquad\abs{K(x,y)}\gtrsim r^{-d}.
	\]
\end{defn}
For bilinear SIOs with variable kernel we set the following.
\begin{defn}\label{defn:nondeg:bil1} A kernel $K:(\R^d)^3\setminus\Delta\to\C$ is said to be non-degenerate if both of the following items hold:
\begin{enumerate} 
	\item[$(1)$] for all points $y,$ there exists two points $x,z$ such that
	\begin{align*}
		\max_{a,b\in\{x,y,z\}}\abs{a-b} > r,\qquad	\abs{K(x,y,z)} \gtrsim r^{-2d}.
	\end{align*} 
	\item[$(2)$] for all points $z,$ there exists two points $x,y$ such that
	\begin{align*}
	\max_{a,b\in\{x,y,z\}}\abs{a-b} > r,\qquad	\abs{K(x,y,z)} \gtrsim r^{-2d}.
	\end{align*}
\end{enumerate}
\end{defn}

It is immediate from the size estimate that if  $y,r$ and then $x,z$ are given as in the item $(1)$ of Definition \ref{defn:nondeg:bil1} then $\max_{a,b\in\{x,y,z\}}\abs{a-b} \sim r$. Indeed, to see this, we simply check that 
\begin{align}\label{size}
	r^{-2d} \lesssim \abs{K(x,y,z)} \lesssim (\abs{x-y} + \abs{x-z})^{-2d} \lesssim \max_{a,b\in\{x,y,z\}}\abs{a-b}^{-2d}\lesssim r^{-2d},
\end{align}
which shows the claim. 

\begin{rem} We will use the assumption $(1)$ to prove Theorem \ref{thm:main} for the index $i=1$ and the assumption $(2)$ for the index $i=2.$  It follows that we will run the proofs of all of our results through with the assumption $(1)$ of Definition \ref{defn:nondeg:bil1} and it is clear how to modify them to get the case $i =2.$
\end{rem}

For kernels of rough SIOs we set the following.
\begin{defn}\label{defn:nondeg:bil:rough} A kernel 
	$K_{\Omega}$ is non-degenerate if $\Omega\not= 0,$ i.e. it has at least one non-zero Lebesgue point $\theta = (\theta_1,\theta_2)\in\mathbb{S}^{2d-1}.$
\end{defn}

\begin{rem} The kernel of the bilinear Riesz transform $\calR_i$
	satisfies both items $(1)$ and $(2)$ in Definition \ref{defn:nondeg:bil1} and is also non-degenerate when considered as a rough bilinear SIO as in Definition \ref{defn:nondeg:bil:rough}.
\end{rem}

In \cite{Li2020multilinear} the following definition of non-degeneracy is given; to contrast it with the  non-degeneracy we name it the strong non-degeneracy.
\begin{defn}\label{defn:nondeg:bilstrong} A kernel $K:(\R^d)^3\setminus\Delta\to\C$ is said to be strongly non-degenerate if for each given point $y\in\R^d$ and $r>0$ there exists a point $x\in B(y,r)^c$ such that 
	\[
	\abs{K(x,y,y)} \gtrsim r^{-2d}.
	\]
\end{defn}
It is straightforward that strong non-degeneracy is stronger than non-degeneracy.
\begin{prop}\label{prop:comparison2} Let $K$ be a strongly non-degenerate kernel. Then, the kernel $K$ is  non-degenerate. 
\end{prop}
\begin{proof} We only show the point $(1)$ from Definition $\ref{defn:nondeg:bil1}.$ Fix a point $y\in \R^d,$ then by strong non-degeneracy there exists a point $x\in B(y,r)^c$ so that $\abs{K(x,y,y)} \gtrsim r^{-2d}.$
	Consequently, it remains to write the previous estimate as $\abs{K(x,y,z)} \gtrsim r^{-2d},$ 
	where $z = y$ and to notice that $x\in B(y,r)^c.$
\end{proof}

\begin{defn} We say that a bilinear SIO $T$ is non-degenerate if its kernel $K$ is non-degenerate. Similarly, a bilinear CZO is non-degenerate, if it is bilinear non-degenerate SIO that satisfies at least one of the properties $(1),(2),(3)$ as in Definition \ref{defn:CZO:multilinear}.
\end{defn}

\section{Bilinear approximate weak factorization}
When proving the commutator lower bounds we do not need the full strength of the kernel assumption \eqref{kernel:regularity} and we will replace this with the following weaker assumption: the function $\omega$ satisfies $\omega(0) = 0$,  is increasing, is subadditive and such that
\begin{equation}\label{kernel:dini}
\abs{G(x,y,z) - G(x',y,z)} \leq \omega\big(\frac{\abs{x-x'}}{\abs{x-y}+\abs{x-z}}\big) (\abs{x-y}+\abs{x-z})^{-2d}
\end{equation}
for $G\in \{K,K^{1*},K^{2*}\},$ whenever $\abs{x-x'} \leq \frac{1}{2}\max(\abs{x-y},\abs{x-z}).$ 
Another strengthening is in that the awf argument only requires us to consider the following off-support information on the kernel $K,$
\begin{align*}
T(f1_{Q^0},g1_{Q^2})(x)=\int_{Q^2}\int_{Q^0}K(x,y,z)f(x)g(z)\ud x \ud z,\qquad y\in Q^1,
\end{align*}
where $Q^0,Q^1,Q^2$ are cubes of the same size such that $\max_{i=0,2}\dist(Q^i,Q^1)\sim\ell(Q^1).$
To press the point, there is no reference whatsoever to the operator $T$ and everything is defined with the kernel $K$ only. 

We move to prove the main technical Propositions \ref{prop:bootstrap} and \ref{prop:wf}.
Recall that we only need the assumption $(1)$ from Definition \ref{defn:nondeg:bil1} to show the lower bounds for the commutator $[b,T]_1.$ 
In the following we always work with three cubes $Q^0,Q^1,Q^2$ and variables are reserved to be used as follows, $x\in Q^0, y\in Q^1, z\in Q^2.$

\begin{defn}
	A dyadic grid on $\R^d$ is a collection $\calD$ of cubes satisfying the following.
	\begin{enumerate}[(i)]
		\item For each $k\in\Z$ the collection $\calD_k = \big\{ Q\in\calD: \ell(Q) = 2^k\big\}$ is a disjoint cover of $\R^d.$
		\item For $Q,P\in\calD,$ we have $Q\cap P\in \big\{Q,P,\emptyset\big\}.$
	\end{enumerate}
\end{defn}

\begin{prop}\label{prop:bootstrap} Let $K$ either
	\begin{enumerate}
		\item be a non-degenerate bilinear kernel that satisfies the estimates \eqref{kernel:size}, \eqref{kernel:dini}, or
		\item be a rough non-degenerate bilinear  kernel. 
	\end{enumerate}
	Let  $Q^1\subset \R^d$ be a cube with centre point $c_{Q^1}$ and let $\calD^0$ and $\calD^2$ be arbitrary dyadic grids. 
	
	Then, there exists a constant $A\geq 3,$ cubes $Q^i\in\calD^i$ and points $c_{Q^i}\in Q^i,$ $i=1,2,$ so that the following items hold.
	\begin{enumerate}[(i)]
		\item 	The cubes are separated and have size as follows
		\begin{align}\label{cube:dist}
		\max_{a\in\{0,2\}}\abs{c_{Q^a}-c_{Q^1}} \sim A\ell(Q^1),\qquad \ell(Q^0)\sim\ell(Q^1)\sim\ell(Q^2). 
		\end{align}
		\item There holds that 
		\begin{align}\label{cube:estimate4}
		\abs{K(c_{Q^0},c_{Q^1},c_{Q^2})} \sim A^{-2d}\abs{Q^1}^{-2}.
		\end{align}
		\item For all $x\in Q^0, y\in Q^1,z\in Q^2,$ There holds that 
		\begin{align}\label{cube:estimate3}
		\int_{Q^0}\int_{Q^2} 	\abs{K(x,y,z) - K(c_{Q^0},c_{Q^1},c_{Q^2})} \ud x \ud z \lesssim \omega(A^{-1})A^{-2d},
		\end{align}
		where $\omega(A^{-1})\to 0$ as $A\to\infty.$
		\item 	 There holds that 
		\begin{align}\label{cube:estimate2}
		\Babs{\int_{Q^0}\int_{Q^2} K(x,y,z)\ud x\ud z}\sim \int_{Q^0}\int_{Q^2}  \abs{K(x,y,z)}\ud x\ud z \sim A^{-2d}.
		\end{align}
	\end{enumerate}

    Moreover, the similar estimates to \eqref{cube:estimate3} and \eqref{cube:estimate2} where we always integrate over any two of the cubes $Q^0,Q^1,Q^2$  with the corresponding variables $x,y,z,$ hold. 
\end{prop}

\begin{rem} We only need to set up $Q^i\in\calD^i,$ $i=0,2,$ for the study of the super-diagonal case $r=1.$
\end{rem}

\begin{proof}[Proof of the case $(1)$:]
As we will mostly manage without the property $Q^i\in\calD^i$, we first find any two cubes $Q^i,$ $i=1,2,$ satisfying the rest of the claims.

	We fix a cube $Q^1\subset\R^d$ and denote its centre point with $c_{Q^1}.$ Let $r= A\diam(Q^1)/2$ and by non-degeneracy find two points $c_{Q^0},c_{Q^2}$ so that $c_{Q^2}\in B(c_{Q^1},r)^c$ (the case $c_{Q^0}\in B(c_{Q^1},r)^c$ is completely symmetric) and
	\begin{align}\label{star1}
			\abs{K(c_{Q^0},c_{Q^1},c_{Q^2})} \overset{*}{\sim} r^{-2d} \sim A^{-2d}\abs{Q^0}^{-2}.
	\end{align} 
	The fact that we have $\sim$ above where indicated by $*$ follows from the discussion after Definition \ref{defn:nondeg:bil1}, see line \eqref{size}. Hence the claim \eqref{cube:estimate4} holds.
Then, we let
$$
Q^0 = (c_{Q^0}-c_{Q^1})+ Q^1,\qquad Q^2 = (c_{Q^2}-c_{Q^1})+ Q^1
$$ 
be the cubes respectfully with the centre points $c_{Q^0}$ and $c_{Q^2}.$  Then, it is clear that the claims on the line \eqref{cube:dist} hold.

Towards the remaining two claims, we first estimate
\begin{align*}
\abs{K(x,y,z) - K(c_{Q^0},c_{Q^1},c_{Q^2})} &\leq \abs{K(x,y,z) - K(c_{Q^0},y,z)} \\
&\qquad+ \abs{K(c_{Q^0},y,z) - K(c_{Q^0},c_{Q^1},z)} \\
&\qquad\qquad+ \abs{K(c_{Q^0},c_{Q^1},z) - K(c_{Q^0},c_{Q^1},c_{Q^2})}.
\end{align*}
Then, as for all points $x\in Q^0, y\in Q^1,z\in Q^2,$  we have
$$\abs{x-c_{Q^0}} \leq \frac{1}{2}\abs{c_{Q^0}-z} \leq \frac{1}{2}\max(\abs{c_{Q^0}-y},\abs{c_{Q^0}-z}),$$ 
(which follows immediately by $c_{Q^2}\in B(x,Ar)^c,$ $z\in Q^2$ and that $A\geq 3$), 
the regularity estimate \eqref{kernel:dini} is applicable and we estimate first of the three intermediate terms as 
\begin{align*}
\abs{K(x,y,z) - K(c_{Q^0},y,z)} &\lesssim \omega\big(\frac{\abs{x-c_{Q^0}}}{\abs{c_{Q^0}-z}+\abs{c_{Q^0}-y}}\big)(\abs{c_{Q^0}-z}+\abs{c_{Q^0}-y})^{-2d} \\
&\lesssim \omega\big(\frac{1/2\diam(Q^0)}{A/3\diam(Q^0)}\big)\big(A\diam(Q^0)\big)^{-2d} \lesssim \omega(A^{-1})A^{-2d}\abs{Q^1}^{-2},
\end{align*}
where in the last estimate we used the sub-additivity of $\omega$.

The remaining two terms estimate similarly and consequently we find that 
\begin{align}\label{hjk}
	\abs{K(x,y,z) - K(c_{Q^0},c_{Q^1},c_{Q^2})}  \lesssim \omega(A^{-1})A^{-2d}\abs{Q^1}^{-2}.
\end{align}
Now, by choosing $A$ large enough, subtracting and adding $K(c_{Q^0},c_{Q^1},c_{Q^2})$ and using \eqref{star1} and \eqref{hjk}
we actually find that 
\begin{align}\label{pert}
	\abs{K(x,y,z)}\sim A^{-2d}\abs{Q^1}^{-2},
\end{align}
which is an improvement of \eqref{cube:estimate4}.
Similarly, by using the estimates \eqref{star1} and \eqref{hjk}, the claims \eqref{cube:estimate3} and \eqref{cube:estimate2} follow immediately.

We still need to argue that we can arrange $Q^i\in \calD^i,$ for $i=0,2.$ Assume that we have shown the claims for the triple of cubes $\wt Q^0, Q^1,\wt Q^2$ with centre points $c_{\wt Q^0},c_{ Q^1},c_{\wt Q^2.}$ Then, we let $Q_d^i\in\calD^i$ be the largest dyadic cube such that $Q_d^i\subset \wt Q^i.$ Now the cubes $Q_d^i$ clearly satisfy the claims on the line \eqref{cube:dist}, and as \eqref{pert} is valid especially with the triple of points $(c_{Q_d^0},c_{Q^1},c_{Q_d^2}),$ we find
\begin{align}\label{true1}
	\abs{K(c_{Q_d^0},c_{Q^1},c_{Q_d^2})}\sim A^{-2d}\abs{Q^1}^{-2} \sim A^{-2d}\abs{Q_d^1}^{-2}
\end{align}
and hence \eqref{cube:estimate4} is checked. The remaining claims are similarly immediate (with $Q_d^i$ in place of $Q^i$) and follow as before.

\end{proof}

\begin{proof}[Proof of the case $(2)$:] We first check the claims with balls in place of cubes. By the non-degeneracy assumption let $\theta=(\theta_0,\theta_2)\in\mathbb{S}^{2d-1}$ be a non-zero Lebesgue point of $\Omega.$ Then, fix a ball $B^1$ with centre $c_{B^1}$ and radius $r.$ Let the points $c_{B^0},c_{B^2}$ be defined by the following identities
	$$	
	c_{B^0} -  c_{B^1}  =  rA\theta_0, \qquad c_{B^0}-c_{B^2} = rA\theta_2,
	$$
	and let $B^i$ be a ball with centre $c_{B^i}$ and radius $r.$ 
	It is then clear that \eqref{cube:dist} holds and that 
	\begin{align*}
			K_{\Omega}(c_{B^0},c_{B^1},c_{B^2}) &=  \frac{\Omega(c_{B^0}-c_{B^1},c_{B^0}-
			c_{B^2})}{\abs{(c_{B^0}-c_{B^1},c_{B^0}-c_{B^2})}^{2d}}  \\ 
		&=\frac{\Omega((rA\theta_0,rA\theta_2)')}{\abs{(rA\theta_0,rA\theta_2)}^{2d}} \sim A^{-2d}\abs{B^1}^{-2d}\abs{\Omega(\theta_0,\theta_1)},
	\end{align*}
	hence \eqref{cube:estimate4} holds. 
	
	It remains to check \eqref{cube:estimate3} and \eqref{cube:estimate2}. Let $x\in B^0,y\in B^1,z\in B^2$ be arbitrary and write
	\[
	x = c_{B^0} + ru_x,\quad y = c_{B^0} - rA\theta_0 + ru_y,\quad z =  c_{B^0}-rA\theta_2 + u_z,
	\]
	for a specific $u_a\in B(0,1)$ depending on the parameter $a\in\{x,y,z\}.$ To ease notation we write $\Omega(h') = \Omega(h)$ and $K_{\Omega} = K.$ Then, we have
	\begin{align*}
		&K(x,y,z)-K(c_{B^0},c_{B^1},c_{B^2}) = \frac{\Omega(x-y,x-
			z)}{\abs{(x-y,x-z)}^{2d}} - \frac{\Omega(c_{B^0}-c_{B^1},c_{B^0}-
			c_{B^2})}{\abs{(c_{B^0}-c_{B^1},c_{B^0}-c_{B^2})}^{2d}} \\
		&= \frac{\Omega\big(rA\theta_0 + r(u_x-u_y),rA\theta_2 + r(u_x-u_z)\big)}{\babs{\big(rA\theta_0 + r(u_x-u_y),rA\theta_2 + r(u_x-u_z)\big)}^{2d}}
		- \frac{\Omega\big(  rA\theta_0,Ar\theta_2\big)}{\babs{rA\theta_0,rA\theta_2}^{2d}} \\
		&= \big(rA\big)^{-2d}\Big(  \Omega\big(  \theta_0 + \frac{u_x-u_y}{A},\theta_2 + \frac{u_x-u_z}{A}  \big)\Babs{\big(  \theta_0 + \frac{u_x-u_y}{A},\theta_2 + \frac{u_x-u_z}{A}  \big)}^{-2d}-\Omega(\theta_0,\theta_2)\Big) \\
		&=  \big(rA\big)^{-2d}\big( I+ II\big),
			\end{align*}
where 
\begin{align*}
	I =  \Big(\Omega\big(  \theta_0 + \frac{u_x-u_y}{A},\theta_2 + \frac{u_x-u_z}{A}  \big)-\Omega(\theta_0,\theta_2) \Big)\Babs{\big(  \theta_0 + \frac{u_x-u_y}{A},\theta_2 + \frac{u_x-u_z}{A}  \big)}^{-2d}
\end{align*}
and 
\begin{align*}
	II = \Omega(\theta_0,\theta_2)\Big(	  \babs{\big(  \theta_0 + \frac{u_x-u_y}{A},\theta_2 + \frac{u_x-u_z}{A}  \big)}^{-2d}-1 \Big).
\end{align*}
With a choice of $A$ large enough we find that
\begin{align*}
	\abs{II} &\leq \abs{\Omega(\theta_0,\theta_2)}\babs{1-\babs{\big(  \theta_0 + \frac{u_x-u_y}{A},\theta_2 + \frac{u_x-u_z}{A}  \big)}^{2d}} \\ 
	&\lesssim_{\Omega(\theta_0,\theta_2)} \babs{\abs{(\theta_0,\theta_2)}^{2d}-\babs{\big(  \theta_0 + \frac{u_x-u_y}{A},\theta_2 + \frac{u_x-u_z}{A}  \big)}^{2d}} \\
	&\overset{*}{\lesssim }\babs{\abs{(\theta_0,\theta_2)}-\babs{\big(  \theta_0 + \frac{u_x-u_y}{A},\theta_2 + \frac{u_x-u_z}{A}  \big)}} \leq \babs{\big( \frac{u_x-u_y}{A},\frac{u_x-u_z}{A}  \big)} \lesssim A^{-1},
\end{align*}
where as indicated by $*$ the mean valued theorem was applied with $x\mapsto x^{2d}$ and we used the estimate $\abs{u_x-u_y}+\abs{u_x-u_z} \lesssim 1.$ Hence, we find that 
\begin{align}\label{hjkl}
	\int_{B^0}\int_{B^2}\left(rA\right)^{-2d}\abs{II}\ud x \ud z \lesssim \omega_{II}(A^{-1})A^{-2d},\qquad \omega_{II}(A^{-1})= A^{-1}.
\end{align}
With a fixed $y,$ the point $u_x-u_y$ varies over $B(0,2)$ and with a fixed $y,x$ the point $u_x-u_z$ varies over $B(0,2).$ Hence, we estimate
\begin{equation}\label{anytwo}
	\begin{split}
	\int_{B^0}\int_{B^2}\left(Ar\right)^{-2d}\abs{I}\ud x \ud z &\lesssim A^{-2d}\fint_{B^0}\fint_{B^2} \babs{\Omega\big(  \theta_0 + \frac{u_x-u_y}{A},\theta_2 + \frac{u_x-u_z}{A}  \big)-\Omega(\theta_0,\theta_2) } \ud x \ud z  \\
	&\leq  A^{-2d}\fint_{B^0}\fint_{B(0,2)} \babs{\Omega\big(  \theta_0 + \frac{u_x-u_y}{A},\theta_2 + \frac{t}{A} \big)-\Omega(\theta_0,\theta_2) } \ud x \ud z  \\
	&\leq A^{-2d}\fint_{B(0,2)}\fint_{B(0,2)}\babs{\Omega\big(  \theta_0 + \frac{s}{A},\theta_2 + \frac{t}{A}  \big)-\Omega(\theta_0,\theta_2) } \ud s\ud t \\
	&=A^{-2d}\fint_{B(0,\frac{2}{A})}\fint_{B(0,\frac{2}{A})}\babs{\Omega\big(  \theta_0 + s,\theta_2 + t  \big)-\Omega(\theta_0,\theta_2) } \ud s\ud t \\ 
	&= A^{-2d}\omega_I(A^{-1}),
	\end{split}
\end{equation}
where 
\[
\omega_I(A^{-1}) = \fint_{B(0,\frac{2}{A})}\fint_{B(0,\frac{2}{A})}\babs{\Omega\big(  \theta_0 + s,\theta_2 + t  \big)-\Omega(\theta_0,\theta_2) } \ud xs\ud t \to 0,\qquad A\to\infty,
\]
by $\theta = (\theta_0,\theta_2)$ being a Lebesgue point of $\Omega.$
Having the preceding estimate together with \eqref{hjkl} shows \eqref{cube:estimate3}, 
	\begin{align*}
&\int_{Q^0}\int_{Q^2} 	\abs{K(x,y,z) - K(c_{B^0},c_{B^1},c_{B^2})} \ud y \ud z \lesssim 	\int_{B^0}\int_{B^2}\left(Ar\right)^{-2d}\abs{I + II}\ud x \ud z  \\ 
 &\lesssim \big( \omega_I(A^{-1})+ \omega_{II}(A^{-1}) \big)A^{-2d} =  \omega(A^{-1})A^{-2d}.
\end{align*}
As before, \eqref{cube:estimate2} follows from \eqref{cube:estimate4} and \eqref{cube:estimate3}.

Lastly, we replace the balls with the desired cubes. Let $Q^1$ be a cube with centre point $c_{Q^1}=c_{B^1}$ such that $Q^1\subset B^1$ for a minimal ball $B^1$ with centre point $c_{B^1},$ and let $Q_d^i$ be the maximal dyadic cubes in $\calD^i$ such that $c_{B^i}\in Q_d^i\subset B^i,$ for $i=1,2.$ We define $c_{Q_d^i}=c_{B^i}$ (these are not necessarily the centre-points).
It is then clear from the setup that the triple of cubes $Q_d^0,Q^1,Q_d^2$ and the points $(c_{Q_d^0},c_{Q^1},c_{Q_d^2})$ satisfy the claims \eqref{cube:dist} and \eqref{cube:estimate4}. Of the remaining claims, the claim \eqref{cube:estimate3} follows, for example, by using the just shown result for balls,
\begin{align*}
&\int_{Q_d^0}\int_{Q_d^2} 	\abs{K(x,y,z) - K(c_{Q_d^0},c_{Q^1},c_{Q_d^2})} \ud x \ud z \\
&\leq \int_{B^0}\int_{B^2} 	\abs{K(x,y,z) - K(c_{B^0},c_{B^0},c_{B^2})} \ud x \ud z \lesssim  \omega(A^{-1})A^{-2d},
\end{align*}
and \eqref{cube:estimate3} together with \eqref{cube:estimate4} implies \eqref{cube:estimate2}. The last claim (we can integrate over any two of the cubes) follows by noting that the estimate for the term $II$ was point-wise and inspecting the estimate \eqref{anytwo} for the term $I.$ 
\end{proof}

From now on whenever we fix a cube $Q^1,$ the associated cubes $Q^0$ and $Q^2$ will stand for the cubes generated through Proposition \ref{prop:bootstrap}.
If a function has support in the cube $Q^i$ then it has the subscript $i$ or $Q^i,$ e.g. if $\supp(g)\subset Q^i,$ then we write $g =g_i = g_{Q^i}.$

\begin{prop}\label{prop:wf} Suppose that $K$ is a non-degenerate bilinear kernel. Then, there exists a large parameter $A$ so that supposing the following items:
	\begin{enumerate}[(i)]
		\item let $Q^1$ be a cube and let $Q^0,Q^2$ stand for the cubes generated by Proposition \ref{prop:bootstrap} above,
		\item let $f$ be a locally bounded function with zero mean supported on the cube $Q^1,$
		\item let $g_i$ be functions such that $\supp(g_i)\subset Q^i$ and 
				$\ave{\abs{g_i}}_{Q^i}\sim\Norm{g_i}{\infty} \gtrsim 1,$
	\end{enumerate}
hold, the function $f$ can be written as 
	\begin{align}\label{decompose:f}
	f = \big[ h_1T^{1*}(g_0,g_2) - g_0 T(h_1,g_2)\big] + \big[ h_0T(g_1,g_2) - g_1T^{1*}(h_0,g_2)\big]  + \wt{f} 
	\end{align}
	and we have the following size and support localization information
	\begin{align}\label{wf:bounds}
	\abs{h_1} \lesssim A^{2d}\abs{f},\qquad \abs{h_0} \lesssim	A^{2d} 	\omega(A^{-1}) \Norm{f}{\infty}\abs{g_0},\qquad 	\abs{\widetilde{f}} \lesssim  \omega(A^{-1})\Norm{f}{\infty} \abs{g_1},
	\end{align}
	where the implicit constants on the line \eqref{wf:bounds} depend only on the implicit constants present in the point $(iii)$ and are otherwise independent of the functions $g_i,$ $i=0,1,2.$
	Moreover, there holds that $\int_{Q^1}\wt{f} = 0.$
\end{prop}

\begin{rem}
	If we were dealing only with the integrability exponents $p,q,r\in (1,\infty)$ then we could choose the appearing functions $g_i$ simply as $1_{Q^i},$ however, due to the fact that we allow $r\in(0,1],$ quite arbitrary functions $g_i$ have to be allowed, see the point $(iii)$ in the statement.
\end{rem}

\begin{proof}
	We  write out the function $f$ as
	\begin{align*}
	f = h_1T^{1*}(g_0,g_2) - g_0 T(h_1,g_2) + \wt{w},\qquad  h_1 =  \frac{f}{T^{1*}(g_0,g_2)},\qquad\wt{w} =   g_0 T(h_1,g_2).
	\end{align*}
	We first check that the function $h_1$ is well-defined. We denote with $K_Q$ the constant $K(c_{Q^0},c_{Q^1},c_{Q^2}),$ where $c_{Q^i}\in Q^i$ are the points as in Proposition \ref{prop:bootstrap}. Let $y\in\supp(f)\subset Q^1$ and split into two parts,
	\begin{align*}
	T^{1*}(g_0,g_2)(y) &= \Big(T^{1*}(g_0,g_2)(y) - K_Q\int_{Q^2}\int_{Q^0}g_0(x)g_2(z)\ud x\ud z  \Big) \\
	&\qquad+  K_Q\int_{Q^2}\int_{Q^0}g_0(x)g_2(z)\ud x\ud z \\ 
	&= I(y) + II.
	\end{align*}
	By the lines \eqref{cube:estimate4} and \eqref{cube:estimate3}, respectfully, of Proposition \ref{prop:bootstrap} we find
\begin{equation}\label{s1}
	\begin{split}
		\abs{I(y)} &\lesssim \int_{Q^2}\int_{Q^0} \babs{K(x,y,z)-K(c_{Q^0},c_{Q^1},c_{Q^2})}\abs{g_0(x)}\abs{g_2(x)}\ud x\ud z \\
	&\lesssim \omega(A^{-1})A^{-2d}\Norm{g_0}{\infty}\Norm{g_2}{\infty}
	\end{split}
\end{equation}
	and 
	\begin{align}\label{s2}
	\abs{II} = \abs{K_Q}\abs{Q^1}^2\abs{\ave{g_0}_{Q^0}}\abs{\ave{g_2}_{Q^2}} \sim A^{-2d}\Norm{g_0}{\infty}\Norm{g_2}{\infty}.
	\end{align}
	Consequently, after choosing $A$ sufficiently large, the estimates \eqref{s1} and \eqref{s2} imply that
	\begin{align}\label{s3}
			\abs{	T^{1*}(g_0,g_2)(y)} \sim A^{-2d}\Norm{g_0}{\infty}\Norm{g_2}{\infty}\gtrsim A^{-2d},
	\end{align} and hence, that the function $h_1$ is well-defined. Then, from \eqref{s3} the first claim on the line \eqref{wf:bounds} is also clear.  Next, we control the term $\wt{w}.$ 
	We expand
	\begin{align*}
	T(h_1,g_2)	= T\Big(h_1 - \frac{f}{K_Q\iint g_0g_2},g_2\Big) +   T\Big(\frac{f}{K_Q\iint g_0g_2},g_2\Big).
	\end{align*}
	We estimate the left term on the right-hand side first and for this fix a point $y\in \supp(h_1 - \frac{f}{K_Q\iint g_0g_2}) = \supp(f) \subset Q^1.$ By the lines \eqref{cube:estimate3} and \eqref{cube:estimate2} of Proposition \ref{prop:bootstrap} we have
	\begin{align*}
	\Babs{h_1(y) - \frac{f(y)}{K_Q\iint g_0g_2}}&= \Babs{f(y)\int_{Q^0}\int_{Q^2}\big(K_{Q}- K(x,y,z)\big)g_0(x)g_2(z)\ud x\ud z} \\
	&\qquad\qquad\qquad\times \Babs{T^{1*}(g_0,g_2)(y)K_Q\iint g_0g_2}^{-1}\\
	&\lesssim \abs{f(y)} \Norm{g_0}{\infty}\Norm{g_2}{\infty} \omega(A^{-1}) A^{-2d} \\ 
	&\qquad\qquad\qquad\times \Big(A^{-2d}\ave{\abs{g_0}}_{Q^0} \ave{\abs{g_2}}_{Q^2} \cdot A^{-2d} \abs{\ave{g_0}_{Q^0}} \abs{\ave{g_2}_{Q^2}} \Big)^{-1} \\
	&\lesssim \omega(A^{-1})A^{2d}\abs{f(y)},
	\end{align*}
	where we used the assumption $(iii).$ 
	Consequently, for $x \in \supp(g_0) \subset Q^0$ there holds that
	\begin{align*}
	\Babs{ T\big(h_1 - \frac{f}{K_Q\iint g_0g_2},g_2\big)(x)} &\leq \int_{Q^1}\int_{Q^2}\abs{K(x,y,z)}\Babs{h_1(y) - \frac{f(y)}{K_Q\iint g_0g_2}}\abs{g_2(z)}\ud y \ud z \\
	&\lesssim \omega(A^{-1})A^{2d}\int_{Q^1}\int_{Q^2}\abs{K(x,y,z)}\abs{f(y)}\abs{g_2(z)}\ud y \ud z  \\
	&\lesssim 	\omega(A^{-1})  \Norm{f}{\infty}.
	\end{align*}
	For the remaining term, we use $\int_{Q^1}f = 0$ to estimate
	\begin{align*}
	\Babs{ T\big(\frac{f}{K_Q\iint g_0g_2},g_2\big)(x)} &= \Babs{ \int_{Q^2}\int_{Q^1}\big(K(x,y,z)-K_Q\big)f(y)g_2(z)\ud y \ud z} \times \Babs{K_Q\iint g_0g_2}^{-1} \\
	&\lesssim   \int_{Q^2}\int_{Q^1}\babs{ \big(K(x,y,z)-K_Q\big)f(y)g_2(z)}\ud y \ud z \\ 
	&\qquad\times \big(\abs{\ave{g_0}_{Q^0}}\abs{\ave{g_2}_{Q^2}}A^{-2d}\big)^{-1} \\
	&\lesssim 	\omega(A^{-1})A^{-2d} \Norm{f}{\infty}\Norm{g_2}{\infty}  \times (\abs{\ave{g_0}_{Q^0}}\abs{\ave{g_2}_{Q^2}}A^{-2d})^{-1} \\
	&\lesssim 	\omega(A^{-1})  \Norm{f}{\infty}.
	\end{align*}
	By having the above two estimates together it follows that 
	\begin{align*}
	\abs{\wt{w}(x)} \lesssim		\omega(A^{-1}) \Norm{f}{\infty} \abs{g_0(x)}.
	\end{align*}
	There holds that
	\begin{align}\label{mean0}
			\int\wt{w} = \int g_0 T\Big(\frac{f}{T^{1*}(g_0,g_2)} ,g_2\Big) = \int \frac{f}{T^{1*}(g_0,g_2)} T^{1*}( g_0 ,g_2) = \int f = 0.
	\end{align}
	All the properties of the function $f$ on the cube $Q^1$ that allowed us to run through the first iteration of the decomposition, are enjoyed by the function $\wt{\omega}$ on the cube $Q^0.$ Also, for the kernel $K^{1*}$ of $T^{1*}$ we have $
	\abs{K^{1*}(c_{Q^1},c_{Q^0},c_{Q^2})} = \abs{K(c_{Q^0},c_{Q^1},c_{Q^2})}\gtrsim r^{-2d}.$ Exchanging the roles of $T$ and $T^{1*}$ we iterate the above once more and we write out the function $\wt{\omega}$ as 
	$$
	\wt{w} = h_0T(g_1,g_2) - g_1T^{1*}(h_0,g_2) + \wt{f},\qquad h_0 = \frac{\wt{w}}{T(g_1,g_2)},\qquad \wt{f} = g_1T^{1*}(h_0,g_2).
	$$
	Repeating the above arguments, we find that
	\begin{align*}
	\abs{h_0} \lesssim A^{2d} \abs{\wt{w}} \lesssim	A^{2d} 	\omega(A^{-1}) \Norm{f}{\infty}\abs{g_0}
	\end{align*}
	and
	\begin{align*}
			\abs{\wt{f}} \lesssim \omega(A^{-1}) \Norm{\wt{w}}{\infty}\abs{g_1} \lesssim\omega(A^{-1})^2\Norm{f}{\infty}\abs{g_1}.
	\end{align*}
	Consequently, we have checked the remaining claims on the line \eqref{wf:bounds} and it remains to check  that $\int\wt{f}= 0,$ however, this follows by using the adjoints similarly as it did for the function $\wt{w}$ on the line \eqref{mean0}.
\end{proof}

In the remaining propositions of this section we relate the oscillation to commutator norms. Recall, that
the oscillation of a function $b\in L^1_{\loc}$ over a cube $Q$ is 
	$$\osc(b;Q) = \fint_Q\abs{b-\ave{b}_Q}.$$
Also, let $\gamma \in (0,1).$ Then, a subset $F'\subset F$ is said to be a $\gamma$-major subset, if $\abs{F'}>\gamma\abs{F}.$

\begin{prop}\label{prop:osc1:bound}  Suppose that $K$ is a bilinear non-degenerate kernel, $b\in L^1_{\loc}$ and $\gamma\in(0,1).$ Fix a cube $Q^1$ and let $g_i = 1_{E_{Q^i}},$ for $i=0,1,2,$ where $E_{Q^i}\subset Q^i$ is a $\gamma$-major subset. 
	Then, there holds that
	 	\begin{align}\label{osc1:bound}
	 \abs{Q^1}\osc(b;Q^1) \lesssim \abs{\big\langle [b,T]_1(h_1,g_2), g_0 \big\rangle}  + \abs{\big\langle [b,T]_1(g_1,g_2), h_0 \big\rangle},
	 \end{align}
	 and we have the following size and support localization information,
	\begin{align}\label{info:size1}
	\abs{h_1} \lesssim 1_{Q^1},\qquad \abs{h_0} \lesssim \omega(A^{-1})\abs{g_0},
	\end{align}
	where the implicit constants depend only on $\gamma.$
\end{prop}

\begin{proof} By $b-\ave{b}_{Q^1}$ having zero mean on the cube ${Q^1}$ and duality find a function $f$ with the properties $\int_{Q^1}f = 0,$ $\Norm{f}{L^{\infty}} \leq 2$ such that
	\[
	 \abs{{Q^1}}\osc(b;{Q^1}) = \int_{Q^1}  bf.
	\]
By Proposition \ref{prop:wf} we write out the function $f$ to arrive at  
\begin{align*}
		\int_{Q^1}  bf &= \int b  \big[ h_1T^{1*}(g_0,g_2) - g_0 T(h_1,g_2)\big] + \int b\big[ h_0T(g_1,g_2) - g_1T^{1*}(h_0,g_2)\big]  + \int_{Q^1}b\wt{f} \\ 
		&= -\big\langle [b,T]_1(h_1,g_2), g_0 \big\rangle  - \big\langle [b,T]_1(g_1,g_2), h_0 \big\rangle + \int_{Q^1}b\wt{f}.
\end{align*}
The claims on the line \eqref{info:size1} follow immediately from $\Norm{f}{L^{\infty}} \leq 2,$ the choice of the functions $g_i$ and the corresponding information in Proposition \ref{prop:wf}. Then,
by $\int\wt{f} = 0$ and the bound \eqref{wf:bounds} on the error term, we estimate
\[
\Babs{\int_{Q^1}b\wt{f}} = \Babs{\int_{Q^1}(b-\ave{b}_{Q^1})\wt{f}} \leq \Norm{\wt{f}}{L^{\infty}}\abs{{Q^1}}\osc(b;{Q^1}) \lesssim \omega(A^{-1})\abs{{Q^1}}\osc(b;{Q^1}).
\]
Consequently, we find that
\[
	\abs{{Q^1}}\osc(b;{Q^1}) \lesssim \abs{\big\langle [b,T]_1(h_1,g_2), g_0 \big\rangle}  + \abs{\big\langle [b,T]_1(g_1,g_2), h_0 \big\rangle} + \omega(A^{-1})\abs{{Q^1}}\osc(b;{Q^1}).
\]
Now, as $b\in L^1_{\loc},$ the common term shared on both sides of the estimate is finite, and hence, by choosing $A$ large enough, we absorb it to the left-hand side and the claim follows.
\end{proof}
The off-support norms that model the commutator norm will be given next.
When $1<r<\infty$ we will use the following off-support norm.
\begin{defn}\label{norm:diag} Let $p,q,r\in(1,\infty).$ Let $K$ be a kernel that is locally bounded outside the diagonal $\Delta$ and let $b\in L^1_{\loc}.$ Then, we define the off-support norm 
	\begin{align*}
	\mathcal{O}^A_{p,q,r}(b;K) &= \sup\Babs{\int_{Q^0}\int_{Q^2}\int_{Q^1}(b(x)-b(y))K(x,y,z)f_1(y)f_2(z) f_0(x)\ud y\ud z\ud x}\\
	&\qquad\times \abs{Q^0}^{-(1/p+1/q+1/r')},
	\end{align*}
	where the supremum is taken over all triples of cubes $Q^0,Q^1,Q^2$ of the same size such that $\max_{a,b\in\{0,1,2\}} \dist(Q^a,Q^b) \sim A\diam(Q^0)$
	and over all functions $f_a$ such that $\abs{f_a} \leq 1_{Q^a}$ for $a\in\{0,1,2\}.$
\end{defn}

\begin{rem}
	It is immediate from H{\"o}lder's inequality that 
	$$
	\mathcal{O}^A_{p,q,r}(b;K) \leq \Norm{[b,T]_1}{L^p\times L^q \to L^r},
	$$
	whenever $K$ is the kernel of $T$. For example, when $r=1$ and $r' = \infty$, we have
	\begin{align*}
	\mathcal{O}^A_{p,q,r}(b;K) &= \sup\babs{\langle f_0 1_{Q^0}, [b,T]_1(f_1,f_2)\rangle}\times \abs{Q^1}^{-(1/p+1/q)} \\
	&\leq  \sup \Norm{f_0}{\infty}\Norm{[b,T]_1(f_1,f_2)}{L^1} \abs{Q^1}^{-(1/p+1/q)} \leq  \Norm{[b,T]_1}{L^p\times L^q\to L^1}.
	\end{align*}
\end{rem}
When $0<r<1$ we will use the following off-support norm.
\begin{defn} \label{norm:quasi:offspt} Let $r\in(0,\infty)$ and $p,q\in(1,\infty),$ let $K$ be any kernel that is locally bounded outside the diagonal and let $b\in L^1_{\loc}.$ 
	Then, we define the weak off-support norm $\mathcal{O}^{\infty,A}_{p,q,r}(b;K) $ to be the smallest constant $C$ such that for all triples of cubes $Q^0,Q^1,Q^2$  of the same size and functions $f_1,f_2$ satisfying 
	$$
	\underset{a,b\in\{0,1,2\} }{\max}\dist(Q^a,Q^b) \sim A\diam(Q^0),\qquad 	\abs{f_1}\leq 1_{Q^1},\quad \abs{f_2}\leq 1_{Q^2},
	$$
	there exists a major subset $F'\subset Q^0$  such that if $\abs{f_0}\leq 1_{F'},$ then
	\begin{align*}
 \Babs{\int_{Q^0}\int_{Q^2}\int_{Q^1}(b(x)-b(y))K(x,y,z)f_1(y)f_2(z) f_0(x)\ud y\ud z\ud x}\leq C \abs{Q^0}^{1/p+1/q+1/r'}.
	\end{align*}
\end{defn}

We now fix the constant $A$ to be so large that all the above propositions where it appears are applicable. Hence, we will also drop the superscript $A$ from the  off-support norms \ref{norm:diag} and \ref{norm:quasi:offspt} and only write  $\mathcal{O}^{\infty}_{p,q,r}, \mathcal{O}_{p,q,r}.$
 As $\mathcal{O}^{\infty}_{p,q,r}(b;K) \leq \mathcal{O}_{p,q,r}(b;K),$ also
$\mathcal{O}^{\infty}_{p,q,r}$ is a reasonable off-support norm in the Banach range of exponents.
Before connecting the off-support norms  to the commutator, we remark the following a priori upper bound.
\begin{rem}  If $K$ is a bilinear kernel satisfying the size estimate \eqref{kernel:size}, then 
	\[
	\mathcal{O}_{p,q,r}(b;K) \lesssim 		\sup_{Q}\ell(Q)^{\alpha}\fint_Q\abs{b-\ave{b}_Q}.
	\]
	This is quickly seen as follows: fix triples $Q^i,f_i$ for $i\in\{0,1,2\}$ as in the Definition \ref{norm:diag} and let $\wt{Q}$ be a minimal cube such that $Q^0,Q^1\subset \wt{Q}.$ Then, by the triangle inequality we see that it is enough to control two symmetric terms of which the other one is and is controlled as
	\begin{align*}
	&\abs{Q^1}^{-\big(1/p+1/q+1/r'\big)}\int_{Q^0}\int_{Q^2}\int_{Q^1} \abs{b(x)-\ave{b}_{\wt{Q}}}\abs{K(x,y,z)f_1(y)f_2(z) f_0(x)}\ud y\ud z\ud x \\
	&\lesssim\abs{Q^1}^{-\big(1/p+1/q+1/r'\big)} (A\ell(Q^1))^{-2d}\abs{Q^1}^2\int_{Q^0} \abs{b(x)-\ave{b}_{\wt{Q}}}\ud x \\ 
	&\sim \ell(\wt{Q})^{-\alpha} \fint_{\wt{Q}}\abs{b(x)-\ave{b}_{\wt{Q}}}\ud x.
	\end{align*}
\end{rem}

Next, we relate the weak off-support norm $\mathcal{O}^{\infty}_{p,q,r}(b;K)$ to the commutator. For this, recall that a function $f$ belongs to the space $L^{s,\infty}(\R^d),$  $0<s<\infty,$ if 
\[
\Norm{f}{L^{s,\infty}(\R^d)} := \sup_{\lambda>0}\lambda\abs{ \{ x\in\R^d: \abs{f(x)}>\lambda \}}^{1/s} < \infty. 
\]
Also, recall that $\Norm{f}{L^{s,\infty}} \leq \Norm{f}{L^s},$ $s>0.$ The following Lemma \ref{lem:quasidual} is standard, see e.g. Section 2.4. Dualization of quasi-norms in the book \cite{MuscaluSchlag2013vol2} of Muscalu and Schlag. 

\begin{lem}\label{lem:quasidual} Let $s\in(0,\infty)$ and fix a constant $C>0.$ Then, the following are equivalent:
	\begin{enumerate}[(i)]
		\item  There holds that $\Norm{\psi}{L^{s,\infty}(\R^d)} \lesssim C.$
		\item For each set $F$ with $\abs{F}\in(0,\infty)$, there exists a major subset $F'\subset F$ such that for all functions $\abs{g}\leq 1_{F'},$ there holds that $\abs{\langle \psi,g\rangle}\lesssim C\abs{F}^{1/s'}.$
	\end{enumerate}
\end{lem}

Taken together, the following two propositions control the oscillation with the commutator norm.
\begin{prop}\label{prop:norm:weak:bound} Let $p,q,r\in (0,\infty)$ be arbitrary exponents. Then, there holds that 
	\begin{align}\label{xxx}
	\mathcal{O}^{\infty}_{p,q,r}(b;K) \lesssim \Norm{[b,T]_1}{L^p\times L^q\to L^{r,\infty}},
	\end{align}
	whenever $T$ has the kernel $K$ and the commutator is well-defined.
\end{prop}
\begin{proof}
	Consider a triple $Q^0,Q^1,Q^2$ and functions $f_1,f_2$ as in the Definition \ref{norm:quasi:offspt} of\\  $\mathcal{O}^{\infty}_{p,q,r}(b;K).$ Clearly we may assume that the right-hand side of \eqref{xxx} is finite and hence that  $\Norm{[b,T]_1(f_1,f_2)}{L^{r,\infty}}<\infty.$  Then, denote $F = Q^0$ and let $F'\subset F$ be the major subset given by the item $(2)$ in Lemma \ref{lem:quasidual} such that for all functions $\abs{f_0}\leq 1_{F'}$ there holds that
  \begin{align*}
	&\abs{\int_{Q^0}\int_{Q^2}\int_{Q^1}(b(x)-b(y))K(x,y,z)f_1(y)f_2(z) f_0(x)\ud y\ud z\ud x} \\
	& =\abs{\big\langle [b,T]_1(f_1,f_2), f_0 \big\rangle} \lesssim \Norm{[b,T](f_1,f_2)}{L^{r,\infty}}\abs{F}^{1/r'}\lesssim \Norm{[b,T]_1}{L^p\times L^q\to L^{r,\infty}}\abs{Q^0}^{1/p+1/q+1/r'},
	\end{align*}
	which implies the claim.
\end{proof}

\begin{prop}\label{prop:osc1:bound1}Let $p,q,r\in (0,\infty)$ be arbitrary exponents and let $K$  be a bilinear non-degenerate kernel. Then, for all cubes $Q^1\subset\R^d$ there holds that 
	\begin{align*}
		\osc(b;{Q^1})  \lesssim \mathcal{O}^{\infty}_{p,q,r}(b;K) \abs{{Q^1}}^{1/p+1/q-1/r}.
	\end{align*}
\end{prop}
\begin{proof} Fix a cube $Q^1$ and let $Q^0,Q^2$ be the cubes given by Proposition \ref{prop:wf} and let  $g_i = 1_{Q^i},$ $i=1,2.$ Then, according to the definition of $\mathcal{O}^{\infty}_{p,q,r}(b;K)$  let $F'\subset Q^0$ be a major subset and define $g_0 = 1_{F'}.$ Then, by Proposition \ref{prop:osc1:bound}, we find that
	\begin{align*}
	\abs{{Q^1}}\osc(b;{Q^1}) \lesssim     \abs{\big\langle [b,T]_1(h_1,g_2), g_0 \big\rangle}  + \abs{\big\langle [b,T]_1(g_1,g_2), h_0 \big\rangle} \lesssim \mathcal{O}^{\infty}_{p,q,r}(b;K) \abs{{Q^1}}^{1/p+1/q+1/r'},
	\end{align*}
	which is the claim rearranged.
\end{proof}

\section{The cases $r^{-1}\leq \sigma(p,q)^{-1}$}
 In this section we will be either on the diagonal, meaning that $r^{-1} = \sigma(p,q)^{-1},$ or on the sub-diagonal, meaning that $r^{-1} < \sigma(p,q)^{-1}.$ In both cases the lower bounds formulate simultaneously in Theorem \ref{thm:diag/subdiag/lb} and the upper bounds in theorems \ref{thm:diag/ub} and \ref{thm:subdiag/ub}.

\begin{thm}\label{thm:diag/subdiag/lb} Let $b\in L^1_{\loc}$ and let $0<r,p,q<\infty$ be such that  $r^{-1} \leq \sigma(p,q)^{-1}$ and let $\alpha := d\big( \sigma(p,q)^{-1} - r^{-1}\big).$ Suppose that $K$ is a bilinear non-degenerate CZ-kernel, then 
		\begin{align*}
			\sup_{Q}\ell(Q)^{-\alpha}\fint_Q\abs{b-\ave{b}_Q} 
			\lesssim {O}_{p,q,r}^{\infty}(b;K) \lesssim \Norm{[b,T]_1}{L^p\times L^q\to L^{r,\infty}}
		\end{align*}
\end{thm}

\begin{proof} Follows immediately from propositions \ref{prop:norm:weak:bound} and \ref{prop:osc1:bound1}. 
\end{proof}

The following upper bound is well-known and is recorded e.g. in \cite{Li2020multilinear}.
\begin{thm}\label{thm:diag/ub} Let $\frac{1}{2}<r<\infty$ and $1<p,q<\infty$ be such that $\frac{1}{r} = \frac{1}{p}+\frac{1}{q}$ and $T\in\CZO(2,d,\omega).$ Then, $$\Norm{[b,T]_1}{L^p\times L^q\to L^r} \lesssim \Norm{b}{\BMO}.$$
\end{thm}

The sub-diagonal upper bound in Theorem \ref{thm:subdiag/ub} requires some preparation consisting of extending parts from the linear theory to the bilinear setting. We refer the reader to Grafakos \cite{GrafakosMFA} for a complete account of the corresponding linear theory.

\begin{prop}\label{prop:difference} Let $U,T:\Sigma\times\Sigma \to L^1_{\loc}$ be bilinear SIOs with the same kernel.
	Then, there exists a function  $m\in L^1_{\loc}$ so that $(U-T)(f_1,f_2) = mf_1f_2$ for all $f_1,f_2\in\Sigma.$
	
	In addition, if $U,T$ are CZOs (bounded), then the identity $(U-T)(f_1,f_2) = mf_1f_2$ extends to all functions $f_1,f_2\in L^{\infty}_c$ and the function $m$ is bounded.
\end{prop}
\begin{proof}  We will first show the so called consistency condition: Let $Q\subset\R^d$ be a cube and $f_1,f_2\in\Sigma,$ then almost everywhere
	\begin{align}\label{gX}
	(U-T)(1_Qf_1,1_Qf_2) = 1_Q(U-T)(f_1,f_2).
	\end{align}
	We reduce this to two parts, clearly \eqref{gX} follows if we show that 
	\begin{align}\label{gXB}
	(U-T)(f_1,1_Qf_2) = 1_Q(U-T)(f_1,f_2),
	\end{align}
	and
	\begin{align}\label{gXA}
	(U-T)(1_Qf_1,f_2) = 1_Q(U-T)(f_1,f_2).
	\end{align}
	 We only show the claim \eqref{gXA}, the claim \eqref{gXB} being similar.
	
 As the operators $U,T$ share the kernel $K,$ i.e. $U_{\varepsilon} = T_{\varepsilon},$  the claim \eqref{gXA} follows if we show that for $H\in \{U,T\}$ and all points $x\in\R^d,$ there exists $\varepsilon>0$ such that 
	\begin{align}\label{g2}
	(H-H_{\varepsilon})(1_Qf_1,f_2)(x) = 1_Q(x)(H-H_{\varepsilon})(f_1,f_2)(x).
	\end{align}
	
	Assume first that $x\in (Q\cup\partial Q)^c$ (the claim is made modulo sets of measure zero and hence we remove the boundary).  Then, choose $\varepsilon = \frac{1}{2}\dist(x,\partial Q)$ so that for all points $y \in Q$ there holds that
	$
	\max(\abs{x-y},\abs{x-z})\geq \abs{x-y}>\varepsilon.
	$ Consequently, as $x\not\in\supp(1_Qf_1),$ it follows by the definition of $K$ being the kernel of $H$ that
	\begin{align*}
	H(1_Qf_1,f_2)(x) &= \int_{\R^d}\int_{\R^d}K(x,y,z)1_Q(y)_1(y)f_2(z)\ud y\ud z  \\ &= \int\int_{	\max(\abs{x-y},\abs{x-z})>\varepsilon}K(x,y,z)f_1(y)f_2(z)\ud y\ud z = H_{\varepsilon}(1_Qf_1,f_2)(x).
	\end{align*}
	and we find both sides of \eqref{g2} to be zero.
	Then, let $x\in Q\setminus \partial Q$ and again fix  $\varepsilon = \frac{1}{2}\dist(x,\partial Q).$
	Then, as above, we see that $
	(H-H_{\varepsilon})(1_{Q^c}f_1,f_2)(x) = 0.$
	Consequently, for $x\in Q\setminus\partial Q$ there holds that 
	\begin{align*}
	(H-H_{\varepsilon})(1_{Q}f_1,f_2)(x) &= (H-H_{\varepsilon})(1_{\R^d}f_1,f_2)(x) - (H-H_{\varepsilon})(1_{Q^c}f_1,f_2)(x)   \\
	&= (H-H_{\varepsilon})(f_1,f_2)(x) - 0 \\
	&= 1_{Q}(x)(H-H_{\varepsilon})(f_1,f_2)(x).
	\end{align*}
	Hence, the identity \eqref{gXA} holds almost everywhere.
	
	Then, we define the function $m$ by
	\begin{align}\label{welldefined}
	m1_Q(x) = 1_{Q}(x)(U-T)(1_Q,1_Q)(x),\qquad x\in Q
	\end{align}
	and, as the intersection of two cubes is a cube, the property \eqref{gX} shows that this is well-defined. 
	
	Then, let $f_i,$ $i=1,2,$ be simple and let $x\in \R^d.$ Fix a cube $Q$ such that $\supp(f_1)\cup\supp(f_2)\cup\{x\}\subset Q.$ Then, by \eqref{gX} and linearity there holds that 
	\begin{align*}
	(U-T)(1_Qf_1,1_Qf_2)(x) &= f_1(x)(U-T)(1_Q,1_Qf_2)(x) \\
	&= f_1(x)f_2(x)(U-T)(1_Q,1_Q)(x) = f_1(x)f_2(x)m(x).
	\end{align*}
	Consequently, we have shown that $U-T = m$ on $\Sigma\times\Sigma,$ and this also gives $m\in L^1_{\loc}$ by testing against simple functions. 
	
	If $U,T$ are bounded operators (say $L^4\times L^4\to L^2$) the identity  $(U-T)(f_1,f_2) = mf_1f_2$ follows by approximating $L^4$ functions with those in the class $\Sigma$ (for which the identity holds) and as $L^{\infty}_c\subset L^4$ the desired identity follows.  Also, by testing against simple functions, it follows by the boundedness of $U,T$ that necessarily $\Norm{m}{\infty} \leq \Norm{U}{L^4\times L^4\to L^2}+\Norm{T}{L^4\times L^4\to L^2}.$

\end{proof}

\begin{prop}\label{prop:T:pv} Let $K$ be a kernel such that 
$\sup_{\vare>0}\Norm{T_{\vare}}{L^p\times L^q\to L^{\sigma(p,q)}} < \infty$ for some exponents satisfying $1<p,q<\infty$ and $1\leq \sigma(p,q).$
Then, there exists a bounded bilinear operator $T_0:L^p\times L^q\to L^{\sigma(p,q)}$ with the kernel $K$ and a sequence $\varepsilon_k\to 0$ such that 
	\begin{align}\label{weak}
			\lim_{\varepsilon_k\to 0}\bave{T_{\varepsilon_k}(f_1,f_2),f_3} = \bave{T_0(f_1,f_2),f_3},
	\end{align}  for all $f_1,f_2,f_3\in L^{\infty}_c.$ In addition, if $T$ is a CZO with the kernel $K$, then (by Proposition \ref{prop:difference}) there exists a bounded function $m$ such that $T_0 = T + m.$
\end{prop}
\begin{proof} We will show the argument with the exponents $p=q=3$ and $\sigma(p,q) = \frac{3}{2}$.  Let $\calF$ be a countable dense subset of $L^3.$ By the bound $\sup_{\varepsilon >0}\Norm{T_{\varepsilon}}{L^3\times L^3\to L^{\frac{3}{2}}}<\infty,$ H{\"o}lder's inequality and a diagonalization argument, we find a sequence $\varepsilon_k\to 0$ such that for all  $f_1,f_2\in \calF,$
	\begin{align}\label{aid}
		\Lambda_{f_1,f_2}(f_3) = \lim_{\varepsilon_k\to 0}\bave{T_{\varepsilon_k}(f_1,f_2),f_3}
	\end{align}
	defines a bounded linear functional on $ L^3\cap \calF$ with norm 
	$$
	\Norm{\Lambda_{f_1,f_2}}{L^3\cap \calF\to\C} \leq \sup_{\varepsilon >0}\Norm{T_{\varepsilon}}{L^3\times L^3\to L^{\frac{3}{2}}}\Norm{f_1}{L^3}\Norm{f_2}{L^3}.
	$$  By Cauchy sequences this extends as a bounded linear functional to the whole of $L^3$ and then the Riesz representation theorem gives a function $\psi(f_1,f_2)\in L^{\frac{3}{2}} = \big(L^3\big)^*$ such that
	$$
	\Lambda_{f_1,f_2}(f_3)=\int\psi(f_1,f_2)f_3,\qquad\Norm{\psi(f_1,f_2)}{L^{\frac{3}{2}}}\leq 	\Norm{\Lambda_{f_1,f_2}}{L^3\to\C}.
	$$

	Then, we define the operator $T_0(f_1,f_2) = \psi(f_1,f_2)$ in the dense class $\calF\times \calF$ and clearly $T_0: L^3\cap \calF\times  L^3\cap \calF \to L^{\frac{3}{2}}$ is a bounded bilinear operator with the kernel $K$ that satisfies \eqref{weak} for  functions $f_1,f_2\in\calF$ and $f_3\in L^3.$ 
	Again, by Cauchy sequences $T_0$ extends as a bounded bilinear functional to the whole $L^3\times L^3$ and then it remains to argue that $T_0$ has the kernel $K$ and that \eqref{weak} holds for $f_1,f_2,f_3\in L^3.$
	That $T_0$ has the kernel $K$ follows by how $T_0$ was extended to $L^3\times L^3$ via Cauchy sequences, the kernel representation being valid in $\calF\times\calF$ and the dominated convergence theorem. Similarly we find that \eqref{weak} holds for $f_1,f_2,f_3\in L^3.$ As $L^{\infty}_{c}\subset L^3,$ we are done. 
\end{proof}

\begin{prop}\label{prop:closedform} Let $T$ be a SIO with a kernel $K$ such that $\sup_{\vare>0}\Norm{T_{\vare}}{L^p\times L^q\to L^{\sigma(p,q)}} < \infty$ for some exponents satisfying $1<p,q<\infty$ and $1\leq \sigma(p,q),$ and let  $f_1,f_2\in L^{\infty}_c$ and $b\in \dot C^{0,\alpha}.$ Then, there holds that 
	\begin{equation}\label{claimpw}
	[b,T]_1(f,g)(x) = \int\int (b(x)-b(y))K(x,y,z)f(y)g(z)\ud y \ud z.
	\end{equation}
\end{prop}
\begin{proof} 
	As $b\in \dot C^{0,\alpha}\subset  L^{\infty}_{\loc},$ $bf_1\in L^{\infty}_c.$ Then by Proposition \ref{prop:T:pv} we have
	\begin{equation}\label{now}
	\begin{split}
		\bave{	[b,T]_1(f_1,f_2),f_3} &= 	\bave{	[b,T_0-m]_1(f_1,f_2),f_3} =\bave{	[b,T_0]_1(f_1,f_2),f_3} \\ 
		&= \bave{\lim_{\varepsilon_k\to 0}[b,T_{\varepsilon_k}](f_1,f_2),f_3} \\
	&\overset{*}{=} \bave{\int\int(b(x)-b(y))K(x,y,z)f_1(y)f_2(z)\ud y \ud z,f_3(x)}_x
	\end{split},
	\end{equation}
	where the last step marked with $*$ follows by the dominated convergence theorem after the following estimate (uniform in $\varepsilon_k$)
	\begin{align*}
	\abs{[b,T_{\varepsilon_k}](f_1,f_2)(x)} = &\abs{ \iint_{\max(\abs{x-y},\abs{x-z})>\varepsilon_k}  (b(x)-b(y))K(x,y,z)f_1(y)f_2(z) \ud y \ud z} \\
	&\leq  \int\int \abs{ (b(x)-b(y))K(x,y,z)f_1(y)f_2(z)} \ud y \ud z \\
	& \lesssim \Norm{b}{\dot C^{0,\alpha}(\R^d)}\int\int \frac{\abs{x-y}^{\alpha}}{\big(\abs{x-y} + \abs{x-z}\big)^{2d}}\abs{f_1(y)}\abs{f_2(z)}\ud y \ud z  < \infty,
	\end{align*}
	where the finiteness follows simply by the fact that $f,g\in L^{\infty}_c$ and that the appearing singularity is weak enough to be locally integrable (see also the last estimate in the proof of Theorem \ref{thm:subdiag/ub}). Now as \eqref{now} holds for all test functions $f_3$ the claim on the line \eqref{claimpw} follows.
\end{proof}

\begin{thm}\label{thm:subdiag/ub}  Let $b\in L^1_{\loc},$ let $\frac{1}{2}<r,p,q<\infty$ be such that  $r^{-1} < \sigma(p,q)^{-1},$ let $\alpha := d\big(\sigma(p,q)^{-1}-r^{-1}\big),$ and let  $T$ be a bilinear SIO such that $\sup_{\vare>0}\Norm{T_{\vare}}{L^{p_0}\times L^{q_0}\to L^{\sigma(p_0,q_0)}} < \infty$ for one tuple of exponents $1<p_0,q_0<\infty$ with $\sigma(p_0,q_0)\geq 1.$  Then,
	$$
	\Norm{[b,T]_1}{L^p\times L^q \to L^r} \lesssim \sup_{Q}\ell(Q)^{-\alpha}\fint_Q\abs{b-\ave{b}_Q}.
	$$
\end{thm}

\begin{proof}
	Clearly we may assume that $\Norm{b}{\dot C^{0,\alpha}}<\infty,$ as otherwise the claimed estimate is immediate. By density it is enough to prove the claim for functions $f_1,f_2\in L^{\infty}_c.$ 
	Then, by Proposition \ref{prop:closedform} we write the commmutator in a closed form and estimate it as
	\begin{align*}
	\abs{[b,T]_1(f_1,f_2)(x)} &\lesssim \Norm{b}{\dot C^{0,\alpha}(\R^d)}\int\int \frac{\abs{x-y}^{\alpha}}{\big(\abs{x-y} + \abs{x-z}\big)^{2d}}\abs{f_1(y)}\abs{f_2(z)}\ud y \ud z \\
	&=  \Norm{b}{\dot C^{0,\alpha}(\R^d)}\mathsf{I}^{\alpha}(\abs{f_1},\abs{f_2})(x).
	\end{align*}
	The operator $\mathsf{I}^{\alpha}$ is the multilinear fractional integral of Kenig and Stein, see \cite{KeSt1999}, where its boundedness is fully characterized: it satisfies exactly the claimed estimates.
\end{proof}

\section{The case $r^{-1}>\sigma(p,q)^{-1}$} 
Now we are on the super-diagonal, meaning that  $r^{-1}>\sigma(p,q)^{-1},qq$ and we define the exponent $s$ by
$r^{-1} = \sigma(s,p,q)^{-1}.$ The following proposition shows that the membership of $b\in \dot L^s$ is sufficient for commutator boundedness.
\begin{prop}\label{prop:superdiag:ub} Let $\frac{1}{2}<r<\infty$ and $1<p,q<\infty$ be such that $r^{-1}>\sigma(p,q)^{-1}$ and define the exponent $s$ by $r^{-1} = \sigma(s,p,q)^{-1}$. 
Also, let $T$ be a bounded bilinear operator
$$
T:L^p\times L^q\to L^{\sigma(p,q)},\qquad T:L^{\sigma(s,p)}\times L^q\to L^r.
$$
Then, there holds that $\Norm{[b,T]_1}{L^p\times L^q\to L^r} \lesssim\Norm{b}{\dot L^s}.$

\end{prop}
\begin{proof} We first estimate
	\[
		\Norm{[b,T]_1(f,g)}{L^r} = \Norm{[b-c,T]_1(f,g)}{L^r} \leq  	\Norm{(b-c)T(f,g)}{L^r} +  	\Norm{T((b-c)f,g)}{L^r}.
	\]
By H{\"o}lder's inequality and  the boundedness of $T$ we find that
\[
	\Norm{(b-c)T(f,g)}{L^r}  \leq \Norm{b-c}{L^s}\Norm{T(f,g)}{L^{\sigma(p,q)}} \leq  \Norm{b-c}{L^s}\Norm{T}{L^p\times L^q\to L^{\sigma(p,q)}}\Norm{f}{L^p}\Norm{g}{L^q}.
\]
By the boundedness of $T$ and H{\"o}lder's inequality we have
\begin{align*}
		\Norm{T((b-c)f,g)}{L^r} &\leq \Norm{T}{L^{\sigma(s,p)}\times L^q\to L^r}\Norm{(b-c)f}{L^{\sigma(s,p)}}\Norm{g}{L^q} \\
	&\leq \Norm{T}{L^{\sigma(s,p)}\times L^q\to L^r}\Norm{b-c}{L^{s}}\Norm{f}{L^{p}}\Norm{g}{L^q}. 
\end{align*}
Taking the infimum over all $c\in\C$ shows the claim.
\end{proof}
If $r\geq1,$ then $\sigma(s,p)>1$ and $T\in\CZO(2,d,\omega)$ is bounded as in the assumptions of Proposition \ref{prop:superdiag:ub}.

\begin{defn}We say that a collection of sets $\mathscr{S}$ is $\gamma$-sparse, if there exists a pairwise disjoint collection of $\gamma$-major subsets $\mathscr{S}_E = \{E_Q\subset Q: Q\in\mathscr{S}\}.$ 
\end{defn} 

Our sparse collections will be built by splitting into dyadic scales.   For a cube and a dyadic grid $Q\in\calD,$ we denote $\calD_Q = \{P\in\calD: P\subset Q\}.$
Let $f$ be a locally integrable function and let $Q\in\calD$ be a cube, then we set 
\[
S(f;Q) = \big\{ P \in \calD_Q : P \mbox{ is a maximal cube s.t. }\ave{\abs{f}}_P > 2\ave{\abs{f}}_{Q} \big\},
\]
and form the principal stopping time family $\mathscr{S}\subset\calD_Q$ by
$$\mathscr{S} = \cup_k\mathscr{S}_k,\qquad \mathscr{S}_{k+1} = \cup_{P\in \mathscr{S}_k}S(f;P),\qquad \mathscr{S}_0 = \{Q\}.$$
For an arbitrary collection $\mathscr{S}\subset\calD$ of dyadic cubes and for each $Q\in\mathscr{S}$ we let $\ch_{\mathscr{S}}(Q)$ consist of the maximal cubes $P\in\mathscr{S}$ such that $P\subsetneq Q.$ For a given cube $Q\in\mathscr{S}$ we denote $E_Q = Q\setminus \cup_{P\in\ch_{\mathscr{S}}Q}P$ and for each $Q\in\calD$ we let  $\Pi Q = \Pi_{\mathscr{S}}Q$ denote the minimal cube $P$ in $\mathscr{S}$ such that $Q\subset P$ (on the condition that it exists). With this notation then, $
\ch_{\mathscr{S}}(P) = \{Q\in\mathscr{S}\colon Q\subsetneq P, \Pi Q = P\}.$
We also denote $\ch_{\mathscr{S}}^0(P) = \ch_{\mathscr{S}}(P)$ and $\ch^{k+1}_{\mathscr{S}}(P) = \cup_{Q\in\ch^{k}_{\mathscr{S}}(P)}\ch_{\mathscr{S}}Q.$
Lastly, given a cube $Q\in\calD,$ we denote
\[
\Delta_{Q}f = \sum_{P\in\ch_{\calD}(Q)}\big(\ave{f}_P-\ave{f}_Q\big)1_P.
\]
The following Lemma is recorded e.g. in \cite{HyCom}.
\begin{lem}\label{lem:sparse:family} Fix a cube $Q\in \calD$ and a function $f\in L^1_{\loc}$ supported on the cube $Q.$ Then, the principal stopping time family $\mathscr{S}\subset\calD_Q$  is $\frac{1}{2}$-sparse.
	
If $f\in L^\infty(Q)$ and $\int_Qf = 0,$ then we split the function $f$ according to the partition
$$
\mathscr{S} = \cup_{k=0}^N\mathscr{S}_k,\qquad \mathscr{S}_k = \ch^k_{\mathscr{S}}(Q),
$$ 
where the number $N$ is finite and depends only on $\Norm{f}{L^{\infty}(Q)},$
as
\begin{align}\label{f:decomp}
f = \sum_{k=0}^N\sum_{P\in\mathscr{S}_k}f_P,\qquad f_P = \sum_{\Pi_{\mathscr{S}}Q = P}\Delta_Qf,
\end{align}
 and the functions $f_P$ satisfy:
\begin{enumerate}
	\item[(1)] $\int f_P = 0,$
	\item[(2)] $\Norm{f_P}{\infty} \lesssim \ave{\abs{f}}_P,$
	\item[(3)]  $ \sum_{k=0}^N\sum_{P\in\mathscr{S}_k}\Norm{f_P}{\infty}^s1_P \lesssim  (Mf)^s,$  $s>0.$
\end{enumerate}
\end{lem}

\begin{lem}\label{lem:sparse:reflected}  Let $\mathscr{S}$ be a sparse collection, let $\gamma>0$ and let $\calD$ be a dyadic grid. To each cube $Q\in\mathscr{S}$ associate another cube $\wt{Q}\in\calD$ such that $\dist(Q,\wt{Q}) \leq \gamma\ell(Q)$ and $\ell(\wt{Q})\sim\ell(Q).$
Then, the collection $\wt{\mathscr{S}} = \{\wt{Q}: Q\in\mathscr{S}\}$ is sparse.
\end{lem}
\begin{proof} Let $\wt{P},\wt{H}\in\wt{\mathscr{S}}$ be such that $\wt{H}\subsetneq \wt{P}.$ Then, from that $\dist(H,\wt{H}) \leq \gamma\ell(H)$ and $\ell(H)\lesssim\ell(\wt{H}),$ it follows that there exists a  constant $\beta\sim\gamma$ so that $H\subset\beta\wt{P}.$ Consequently, we find that
	\begin{align*}
		\sum_{\substack{\wt{H}\in\wt{\mathscr{S}} \\ \wt{H}\subsetneq \wt{P}}}\abs{\wt{H}} \leq	\sum_{\substack{\wt{H}\in\wt{\mathscr{S}} \\ H\subsetneq \beta\wt{P}}} \abs{\wt{H}}  \lesssim  \sum_{\substack{H\in\mathscr{S} \\ H\subsetneq \beta\wt{P}}} \abs{H}   \lesssim 	\sum_{\substack{H\in\mathscr{S} \\ H\subsetneq \beta\wt{P}}} \abs{E_H}\leq\abs{\beta\wt{P}} \lesssim \abs{\wt{P}},
	\end{align*}
where we used $\ell(\wt{H}) \lesssim \ell(H)$ in the second estimate and the sparseness of $\mathscr{S}$ in the third and the fourth estimates.
We have shown that the collection $\wt{\mathscr{S}}$ is Carleson and as the Carleson condition is equivalent with sparseness for dyadic collections, for this fact see e.g. the book of Lerner and Nazarov \cite{LernerNazarov2015intuitive}, the claim follows.
\end{proof}

\begin{lem}\label{lem:sparse:bound1} Let $p\in(1,\infty)$ and $\mathscr{S}$ be a sparse collection. Then, for any constants $a_Q$ there holds that 
	\begin{align*}
		\Norm{\sum_{Q\in\mathscr{S}} a_Q1_Q}{L^p} \lesssim \Norm{\sum_{Q\in\mathscr{S}}\abs{a_Q}1_{E_Q}}{L^p}.
	\end{align*}
\end{lem}
\begin{proof} The claim follows by duality and the following estimate
	\begin{align*}
		\big\langle \sum_{Q\in\mathscr{S}} a_Q1_Q ,g \big\rangle \leq  \sum_{Q\in\mathscr{S}}\abs{Q}\abs{a_Q}\ave{\abs{g}}_Q &\lesssim  \sum_{Q\in\mathscr{S}}\abs{E_Q}\abs{a_Q}\ave{\abs{g}}_Q \\ 
		&\lesssim \int Mg  \sum_{Q\in\mathscr{S}}\abs{a_Q}1_{E_Q} \lesssim \Norm{\sum_{Q\in\mathscr{S}} \abs{a_Q}1_{E_Q}}{L^p}\Norm{g}{L^{p'}}.
	\end{align*}
\end{proof}

\begin{defn}\label{norm:superdiag} Let $b\in L^1_{\loc},$ $1\leq r,p,q<\infty$ and let $K$ be locally bounded away from the diagonal. Then, we define the super-diagonal off-support norm 
\allowdisplaybreaks \begin{align*}
\mathcal{O}_{p,q,r}^{\Sigma}(b;K)=\sup&\sum_{k=0}^N\Babs{\int\int\int(b(x)-b(y))K(x,y,z)f_{1,k}(y)f_{2,k}(z)f_{0,k}(x)\ud y\ud z \ud x} \\ &\times\mathsf{N}_{p,q,r'}(\vec{f})^{-1}, 
\end{align*}
where
\begin{align*}
\mathsf{N}_{p,q,r'}(\vec{f}) =  \bNorm{\sum_{k=0}^{N} &\Norm{f_{0,k}}{\infty}1_{\supp(f_{0,k})} }{L^p(\R^d)}\bNorm{\sum_{k=0}^{N} \Norm{f_{1,k}}{\infty}1_{\supp(f_{1,k})} }{L^q(\R^d)} \\ \times& \bNorm{\sum_{k=0}^{N} \Norm{f_{2,k}}{\infty}1_{\supp(f_{2,k})} }{L^{r'}(\R^d)},
\end{align*}
and the supremum is taken over all finite collections of triples of cubes of the same diameter such that $
\max_{i,j\in\{0,1,2\}} \dist(Q^i_k,Q^j_k) \sim A\diam(Q^0_k)$
and over all functions such that $\abs{f_{i,k}} \leq 1_{Q^i_k}$ and $\abs{\supp(f_{i,k})}>0.$
\end{defn}

\begin{rem} If $r>1,$ then we can replace the entries $1_{\supp(f_{i,k})},$ in the three terms of $\mathsf{N}_{p,q,r'}(\vec{f}),$ with $1_{Q^i_k},$  for $i=0,1,2.$ 
\end{rem}

\begin{rem}
For bilinear operators $U$ there holds that
$$
\sum_{i=1}^N \langle U(f_i,g_i), h_i \rangle = \E'\E\Big\langle U\big( \sum_{i=1}^N \eps_i\vare'_i f_i,  \sum_{j=1}^N \eps'_j g_j\big),\sum_{l=1}^N \eps_l h_l \Big\rangle,
$$
where $\varepsilon_i,\varepsilon_i'$ are independent random signs, over some probability spaces with expectations denoted respectively as $\E,\E'$, meaning that $\E \varepsilon_i\vare_j= \E' \vare'_i\vare'_j =  1_{\{i=j\}}(i,j).$
Then, H{\"o}lder's inequality shows that for $r\geq1$ we have  $\mathcal{O}_{p,q,r}^{\Sigma}(b;K)  \leq \Norm{[b,T]_1}{L^p\times L^q\to L^r},$ and consequently, that $\mathcal{O}_{p,q,r}^{\Sigma}$ is a reasonable off-support constant for $r\geq 1.$ 
\end{rem}

\begin{prop}\label{prop:superdiagBanach} Suppose that $K$ is a non-degenerate bilinear kernel, that $b\in L^1_{\loc},$ and $1 \leq r,s,p,q<\infty$ are such that $1/r > 1/p+1/q$ and $r^{-1} = \sigma(s,p,q)^{-1}$.
	Then, there holds that 
	$\Norm{b}{\dot L^s} \lesssim \mathcal{O}^{\Sigma}_{p,q,r}(b;K).$

\end{prop}

\begin{proof} Let  $Q^1$ be an arbitrary dyadic cube and let $\calD$ be a dyadic grid containing the cube $Q^1.$ Fix a constant $M>0$ and let $f$ be a function such that 
	\begin{align}\label{supover}
	1_{Q^1}f = f,\quad \int_{Q^1}f = 0,\quad \Norm{f}{\infty} \leq M,\quad \Norm{f}{L^{s'}}\leq 1.
	\end{align}
	Note that $s>r$ and hence $s,s'>1$ are both in the Banach range of exponents.
	Let $\mathscr{S}\subset\calD_Q$ denote the sparse collection of cubes we obtain through Lemma \ref{lem:sparse:family}.
	Write the function $f$ as on the line \eqref{f:decomp} and by Proposition \ref{prop:wf} factorize each of the terms $f_{P^1},$ $P^1\in\mathscr{S},$ as in \eqref{decompose:f}, to arrive at 
	\[
	f_{P^1} = \big[ h_{P^1}T^{1*}(g_{P^0},g_{P^2}) - g_{P^0} T(h_{P^1},g_{P^2})\big] + \big[ h_{P^0}T(g_{P^1},g_{P^2}) - g_{P^1}T^{1^*}(h_{P^0},g_{P^2})\big]  + \wt{f}_{P^1}, 
	\]
	where we have written $h_i = h_{P^i}$ and $g_i = g_{P^i}.$ Next, we will specify how the cubes $P^0,P^2$ and the functions $g_{P^i}$ are chosen. 
	
	By Proposition \ref{prop:bootstrap} we can assume $P^0,P^2\in\calD.$ Then, by Lemma \ref{lem:sparse:reflected} the collection $\wt{\mathscr{S}^0} = \{P^0: P^1\in\mathscr{S}\}\subset\calD$ is sparse and we will denote the pairwise disjoint major subsets with $E_{P^0}.$  By Proposition \ref{prop:wf} we are free to choose the  functions $g_{P^i}$ under the condition $\ave{\abs{g_i}}_{Q^i}\gtrsim \Norm{g_i}{\infty}\gtrsim 1$ and clearly the following choices suffice,
	\begin{align}\label{choice}
	 g_{P^0} = 1_{E_{P^0}},\qquad 	g_{P^1} = 1_{P^1},\qquad g_{P^2} = 1_{P^2}.
	\end{align}
	Now, the off-support norm only controls finite sums, but the collection $\mathscr{S}$ is potentially infinite. Hence, we empty the collection $\mathscr{S}$ through an increasing chain of finite subcollections $ \mathscr{S}^1\subset\mathscr{S}^2\dots\subset \mathscr{S}.$ 
	Then, we have
	\begin{equation}\label{x}
		\begin{split}
		\abs{\int_Q bf} &\leq 	\lim_{n\to\infty}\sum_{P^1\in\mathscr{S}^n} \abs{\langle [b,T]_1(h_{P^1},g_{P^2}), g_{P^0}\rangle} +\sum_{P^1\in\mathscr{S}^n} \abs{\langle  [b,T]_1(g_{P^1},g_{P^2}), h_{P^0} \rangle } \\ 
		&+ \abs{\int b\wt{f}_{\Sigma}},
		\end{split}
	\end{equation}
	where we denote $\wt{f}_{\Sigma} = \sum_{P^1\in\mathscr{S}}\wt{f}_{P^1}$ and the implicit change of integration and summation is easily checked by the dominated convergence theorem after the subsequent estimates. We first analyse the second sum on the right-hand side of \eqref{x}, the first one being similar. 
	By trilinearity we write
	\begin{align}\label{eq:trilin}
\big\langle  [b,T]_1(g_{P^1},g_{P^2}), h_{P^0} \big\rangle    = \big\langle [b,T]_1( \alpha_{P^1}g_{P^1},\alpha_{P^2}g_{P^2}),  \alpha_{P^0}h_{P^0} \big\rangle
	\end{align}
	for any constants with $\alpha_{P^0}\alpha_{P^1}\alpha_{P^2} = 1$. Hence, by the relation $r^{-1} = \sigma(s,p,q)^{-1},$ we take
	\begin{align}\label{choices2}
	\alpha_{P^0} = \Norm{f_{P^1}}{\infty}^{\frac{s'}{r'}-1},\quad 	\alpha_{P^1} = \Norm{f_{P^1}}{\infty}^{\frac{s'}{p}},\quad  \alpha_{P^2} = \Norm{f_{P^1}}{\infty}^{\frac{s'}{q}},\quad 0 = (\frac{s'}{r'}-1)+\frac{s'}{p} +\frac{s'}{q}.
	\end{align}
Then, with the choices \eqref{choices2}, also using
$\abs{h_{P^0}}\lesssim \Norm{f_{P^1}}{\infty}\abs{g_{P^0}} = \Norm{f_{P^1}}{\infty}1_{E_{P^0}},$ we 
	 find that 
	\begin{align*}
	&\sum_{P^1\in\mathscr{S}^n}\abs{\langle [b,T]_1( \alpha_{P^1}g_{P^1},\alpha_{P^2}g_{P^2}),\alpha_{P^0}h_{P^0} \rangle}   \\ 
	&\qquad\qquad\leq \mathcal{O}^{\Sigma}_{p,q,r}(b;K)\bNorm{\sum_{P^1\in\mathscr{S}^n} \Norm{f_{P^1}}{\infty}^{\frac{s'}{p}}1_{P^1} }{L^p(\R^d)} \bNorm{\sum_{P^1\in\mathscr{S}^n}\Norm{f_{P^1}}{\infty}^{\frac{s'}{q}}1_{P^2} }{L^q(\R^d)} \\
	&\qquad\qquad\qquad\qquad\times \bNorm{\sum_{P^1\in\mathscr{S}^n}\Norm{f_{P^1}}{\infty}^{\frac{s'}{r'}}1_{E_{P^0}} }{L^{r'}(\R^d)} = RHS.
	\end{align*}
Now the proof splits into the cases $r>1$ and $r=1.$ 

We first consider the case $r>1$ in which all the three terms of $RHS$ are estimated similarly. From that $\dist(P^1,P^i) \lesssim \ell(P^i)$ and $\ell(P^1)\lesssim\ell(P^i)$ it follows that there exists an absolute constant $C>0$ such that $CP^i\supset P^1.$ This means that the collections $\{CP^i: P^1\in\mathscr{S}\}$ are sparse with the major subsets $E_{P^1}.$
Hence, by Lemma \ref{lem:sparse:bound1}, for $i\in\{0,1,2\}$ and $v\in(1,\infty)$ and $u\in(0,\infty),$ there holds that
\begin{align*}
\bNorm{\sum_{P^1\in\mathscr{S}^n} \Norm{f_{P^1}}{\infty}^{u}1_{P^i} }{L^{v}(\R^d)} &\leq 	\bNorm{\sum_{P^1\in\mathscr{S}^n} \Norm{f_{P^1}}{\infty}^{u}1_{CP^i} }{L^{v}(\R^d)} \lesssim \bNorm{\sum_{P^1\in\mathscr{S}^n} \Norm{f_{P^1}}{\infty}^{u}1_{E_{P^1}} }{L^{v}(\R^d)} \\
&\leq  \bNorm{\sum_{P^1\in\mathscr{S}^n} \Norm{f_{P^1}}{\infty}^{u}1_{P^1} }{L^{v}(\R^d)} \lesssim \bNorm{ (Mf)^{u} }{L^{v}(\R^d)},
\end{align*}
where in the last estimate we used the point-wise estimate $(3)$ from Lemma \ref{lem:sparse:family}.	 
Now, we find that 
	\begin{align*}
		RHS &\lesssim  \mathcal{O}^{\Sigma}_{p,q,r}(b;K)\bNorm{Mf^{\frac{s'}{p}}}{L^p(\R^d)} \bNorm{Mf^{\frac{s'}{q}}}{L^q(\R^d)}\bNorm{Mf^{\frac{s'}{r'}}}{L^{r'}(\R^d)} \\
		&\lesssim   \mathcal{O}^{\Sigma}_{p,q,r}(b;K)\bNorm{f}{L^{s'}(\R^d)}^{\frac{s'}{p}} \bNorm{f}{L^{s'}(\R^d)}^{\frac{s'}{q}}\bNorm{f}{L^{s'}(\R^d)}^{\frac{s'}{r'}} \\ 
		&\leq\mathcal{O}^{\Sigma}_{p,q,r}(b;K),
	\end{align*}
where we used $s'>1$ in the second estimate for the boundedness of the maximal function.

In the case $r = 1$ the first two terms of $RHS$ estimate the same as in the case $r>1$ and the last term estimates differently
\begin{align*}
	RHS &\lesssim \mathcal{O}^{\Sigma}_{p,q,r}(b;K)\bNorm{\sum_{P^1\in\mathscr{S}^n}\Norm{f_{P^1}}{\infty}^{\frac{s'}{r'}}1_{E_{P^0}} }{L^{r'}(\R^d)} = \mathcal{O}^{\Sigma}_{p,q,r}(b;K)\bNorm{\sum_{P^1\in\mathscr{S}^n}1_{E_{P^0}} }{L^{\infty}(\R^d)} \\
	&\leq\mathcal{O}^{\Sigma}_{p,q,r}(b;K)\Norm{1 }{L^{\infty}(\R^d)} = \mathcal{O}^{\Sigma}_{p,q,r}(b;K),
\end{align*}
the crucial step here was the disjointness of the sets $E_{P^0}.$
	
The just shown estimates also hold for the other term, and as the estimates are uniform in $n$, it follows that
\begin{align*}
\abs{\int bf} &\lesssim 	 \mathcal{O}^{\Sigma}_{p,q,r}(b;K) + \abs{\int b\wt{f}_{\Sigma}}.
\end{align*}
 	By Lemma \ref{lem:sparse:family} and Proposition \ref{prop:wf} we have 
	\begin{align}\label{bound:error}
	\abs{\wt{f}_{\Sigma}} \leq \sum_{P^1\in\mathscr{S}}\Norm{\wt{f}_{P^1}}{\infty}1_{P^1} \lesssim \omega(A^{-1})\sum_{P^1\in\mathscr{S}}\Norm{f_{P^1}}{\infty}1_{P^1} \lesssim\omega(A^{-1})Mf
	\end{align}
	and as also $1_{Q^1}\wt{f}_{\Sigma} = \wt{f}_{\Sigma}$ and $\int_{Q^1} \wt{f}_{\Sigma} = 0,$
the function $\wt{f}_{\Sigma}$ satisfies the conditions on the line \eqref{supover} but now with the additional decay  $\lesssim \omega(A^{-1}).$  Consequently, we conclude
\begin{align}\label{step3}
\sup_{\eqref{supover}}\abs{\int bf} &\lesssim 	 \mathcal{O}^{\Sigma}_{p,q,r}(b;K) + \omega(A^{-1}) \sup_{\eqref{supover}}\abs{\int bf}.
\end{align}
The common term on both sides of the estimate \eqref{step3} is finite (recall that $b\in L^1_{\loc}$ and for each $f$ as in the supremum $\Norm{f}{\infty}<M$), and hence by choosing $A$ sufficiently large, by absorbing the common term to the left-hand side we find that
\begin{align}\label{step4}
\sup_{\eqref{supover}}\abs{\int_Q bf} \lesssim 	 \mathcal{O}^{\Sigma}_{p,q,r}(b;K).
\end{align}	
Then, as $s>1,$ the proof is concluded with exactly the same argument by Riesz' representation theorem as in  \cite{HyCom}. For the convenience of the reader we give the full details. Denote $ L^{\infty}_{c,0} = \{\varphi: \int\varphi = 0, \varphi\in L^{\infty}_c\},$ where $L^{\infty}_c$ denotes bounded and compactly supported functions. As the right-hand side of \eqref{step4} is independent of the cube $Q$ and the constant $M,$ we find that
\begin{align*}
 \Lambda: L^{s'}\cap L^{\infty}_{c,0}\to \C,\qquad \Lambda f = \int bf,\qquad  \Norm{\Lambda}{L^{s'}\cap L^{\infty}_{c,0}\to \C} \lesssim \mathcal{O}^{\Sigma}_{p,q,r}(b;K)
\end{align*}
defines a bounded linear functional in a dense subset of $L^{s'}.$ By density and linearity we find a linear extension $\widehat{\Lambda}:L^s\to\C$ of $\Lambda$ such that $\Norm{\widehat{\Lambda}}{L^{s'}\to \C}\leq \Norm{\Lambda}{L^{s'}\cap L^{\infty}_{c,0}\to \C}.$ By the Riesz representation theorem there exists a function $a$ satisfying $\Norm{a}{L^s}\leq \Norm{\widehat{\Lambda}}{L^{s'}\to\C}$ and $\widehat{\Lambda}f = \int af,$ for all $f\in L^{s'}.$ Especially, as $\widehat{\Lambda}$ extends $\Lambda,$ there holds that 
\begin{align}\label{xx}
	\int bf = \int af,\qquad f\in  L^{s'}\cap L^{\infty}_{c,0}.
\end{align}
Let $\psi_{x}^k = \frac{1_{B(x,k^{-1})}}{\abs{B(x,k^{-1})}}$ be an approximation to identity at the point $x$ and define $\varphi_{x,y}^k = \psi_{x}^k - \psi_{y}^k.$ Then $\varphi_{x,y}^k\in L^{s'}\cap L^{\infty}_{c,0}$ and we find by \eqref{xx} and the Lebesgue differentiation theorem that 
\begin{align*}
	b(x)-b(y) = \lim_{k\to\infty} \int b\varphi_{x,y}^k = \lim_{k\to\infty} \int a\varphi_{x,y}^k = a(x)-a(y).
\end{align*}
It follows that $b = a + c$ for some constant $c,$ and especially that	$\Norm{b}{\dot L^s} \lesssim \mathcal{O}^{\Sigma}_{p,q,r}(b;K).$ We are done.
\end{proof} 

Having propositions \ref{prop:superdiag:ub} and \ref{prop:superdiagBanach} together gives us
\begin{thm}\label{thm:superdiag} Let  $1\leq r,s,p,q<\infty$ be such that $r^{-1} = \sigma(s,p,q)^{-1}$ and let $T$ be a non-degenerate bilinear SIO bounded as
	\[
	T:L^p\times L^q\to L^{\sigma(p,q)},\qquad T:L^{\sigma(s,p)}\times L^q\to L^r.
	\]
	Then, there holds that  
	$$
	\Norm{[b,T]_1}{L^p\times L^q\to L^r} \sim \Norm{b}{\dot L^s}.
	$$
\end{thm}

\section{Extension to multilinear setting} It is straightforward to extend all definitions and results to the multilinear setting. The multilinear extension of Theorem \ref{thm:main} is the following.
\begin{thm}\label{thm:mult} Let $b\in L^1_{\loc}(\R^d;\C),$ let $T$ be a non-degenerate $n$-linear Calderón-Zygmund operator, let $\frac{1}{n}<r<\infty$ and $p_i\in (1,\infty)$ for $i=1,\dots,n.$  Then, there holds that 
	\begin{align*}
	\Norm{[b,T]_i}{\prod_{i=1}^n L^{p_i}\to L^r} \sim	\begin{cases}
	\Norm{b}{\BMO}, & \mbox{if}\quad \frac{1}{r} = \sum_{i=1}^n\frac{1}{p_i} \\
	\Norm{b}{\dot C^{\alpha,0}},\quad \alpha = d\big( \sum_{i=1}^n\frac{1}{p_i}-\frac{1}{r}\big), &\mbox{if}\quad \frac{1}{r} < \sum_{i=1}^n\frac{1}{p_i} \\
	\Norm{b}{\dot{L}^s},\quad \frac{1}{r} = \frac{1}{s}+\sum_{i=1}^n\frac{1}{p_i},  & \mbox{if}\quad \frac{1}{r} > \sum_{i=1}^n\frac{1}{p_i},\quad r\geq 1.
	\end{cases}
	\end{align*}
\end{thm}

In the super-diagonal case we were unable to relax the assumption $r\geq 1$ to $0<r<1$ and it is not clear how this could be done. We pose this as an open question for future research.

\bibliography{references}

\end{document}